\documentclass[12pt,reqno,twoside]{amsart}
\usepackage[T1]{fontenc}
\usepackage{a4wide}
\usepackage{lmodern}
\usepackage{microtype}
\usepackage{amsmath,amssymb,amsthm,mathtools}
\usepackage{adjustbox}
\usepackage{mathrsfs}
\usepackage{enumitem}
\usepackage{booktabs}
\usepackage{bm}
\usepackage{array}
\usepackage{relsize}
\usepackage{tikz-cd}
\usepackage{xcolor}
\usepackage{hyperref}
\usepackage[nameinlink,capitalise,noabbrev]{cleveref}
\hypersetup{colorlinks,linkcolor=teal,citecolor=teal}

\numberwithin{equation}{section}
\newtheorem{theorem}{Theorem}[section]
\newtheorem{proposition}[theorem]{Proposition}
\newtheorem{lemma}[theorem]{Lemma}
\newtheorem{corollary}[theorem]{Corollary}

\theoremstyle{definition}
\newtheorem{definition}[theorem]{Definition}

\newtheorem{remark}{Remark}[section]
\newtheorem{example}{Example}[section]

\DeclareMathOperator{\Herm}{Herm}
\DeclareMathOperator{\Vol}{Vol}

\DeclareMathOperator{\coker}{coker}

\DeclareMathOperator{\rank}{rank}
\DeclareMathOperator{\tr}{tr}

\newcommand{\C}{\mathbb C}
\newcommand{\R}{\mathbb R}

\newcommand{\ddbar}{\bar\partial}
\newcommand{\ii}{\sqrt{-1}}
\newcommand{\Id}{\operatorname{Id}}
\newcommand{\HN}{\mathrm{HN}}

\newcommand{\abs}[1]{\lvert #1\rvert}

\title[Surjectivity of real-linear Cauchy--Riemann operators]{Surjectivity of real-linear Cauchy--Riemann operators: from the minimal Harder--Narasimhan slope to automatic transversality}
\author{Qingchun Ji}
\address{School of Mathematical Sciences, Fudan University, Shanghai 200433, China}
\email{qingchunji@fudan.edu.cn}
\author{Jun Yao}
\address{School of Mathematical Sciences, University of Electronic Science and Technology of China, Chengdu 611731, China}
\email{junyao@uestc.edu.cn}
\date{}

\begin{document}

\begin{abstract}
This paper relates the minimal Harder--Narasimhan slope to the surjectivity of real-linear Cauchy--Riemann operators. We establish a conformally invariant $L^2$ criterion and an asymptotic slope criterion, which yield higher-rank automatic transversality criteria for pseudoholomorphic curves beyond the classical rank-one framework. Applications to pseudoholomorphic spheres in $S^6$ provide quantitative $L^2$ obstructions to the integrability of almost complex structures.
\end{abstract}

\maketitle

\section{Introduction}
Transversality is a fundamental issue in the theory of pseudoholomorphic curves. If
\begin{align*}
u:(M,j)\longrightarrow (X,J)
\end{align*}
is a pseudoholomorphic curve, then the linearization of the nonlinear Cauchy--Riemann operator is a real-linear Cauchy--Riemann operator
\begin{align*}
D_u:W^{1,p}(M,u^*TX)\longrightarrow
L^p(M,\Lambda^{0,1}T^*M\otimes_{\C}u^*TX).
\end{align*}
If $D_u$ is surjective, then $u$ is Fredholm regular. For an immersed curve, the corresponding normal operator governs unparametrized normal deformations. Surjectivity obtained without a generic perturbation of $J$ is commonly referred to as \emph{automatic transversality} (see \cite{MS94,MS12,WendlAT}). For a fixed nonintegrable almost complex structure, however, general methods for establishing transversality in a prescribed homology class are limited (see McDuff--Salamon \cite[p. 40]{MS94}). The aim of this paper is to connect this fixed-$J$ transversality problem with the Harder--Narasimhan geometry of the associated holomorphic bundles.

Two different approaches to the automatic transversality of closed pseudoholomorphic curves are particularly relevant to the present work. For a real $4$-dimensional manifold $X$, the normal bundle of an immersed pseudoholomorphic curve is a holomorphic line bundle. A zero-counting argument based on the similarity principle (see \cite[Theorem $1'$]{HLS}) implies that every real-linear Cauchy--Riemann operator on a holomorphic line bundle $L\to M$ is surjective whenever
\begin{align*}
	\deg L>-\chi(M),
\end{align*}
where $\chi(M)$ is the Euler characteristic of $M$. This gives a purely topological automatic transversality criterion depending only on the degree of the normal bundle and the genus of $M$. The rank-one nature of this argument is essential: in higher rank, surjectivity is sensitive to the holomorphic structure of the bundle and cannot in general be inferred from the genus and total degree alone.

A different approach to transversality is based on $L^2$ estimates. The weighted estimates for half-Dirac operators obtained by Ji--Zhu \cite[Proposition 4.1]{JiZhu} imply a curvature estimate for the weighted adjoint of the Dolbeault operator on closed Riemann surfaces. This is the analytic starting point of the present work. By constructing a suitable weight function from a certain Poisson equation, we also extend the related result \cite[Corollary 1.5]{JiZhu} to general real-linear Cauchy--Riemann operators, which arise naturally as the linearized and normal operators of pseudoholomorphic curves in nonintegrable almost complex manifolds.

Let $M$ be a closed connected Riemann surface, let $E\to M$ be a holomorphic vector bundle, and write
\begin{align*}
D=\ddbar_E+A:W^{1,p}(M,E)\longrightarrow
L^p(M,\Lambda^{0,1}T^*M\otimes E),
\end{align*}
where $\ddbar_E$ is the Cauchy--Riemann operator on the vector bundle $E$ and $A:E\longrightarrow \Lambda^{0,1}T^*M\otimes E$ is a smooth real-linear bundle homomorphism. For a Hermitian metric $h$ on $E$ and a $j$-Hermitian metric $g$ on $M$, define
\begin{align*}
\theta_{E,h,g}(x)
=\lambda_{\min}\bigl(\ii\Lambda_gR^{E,h}(x)\bigr),
\end{align*}
where $R^{E,h}$ is the curvature of the Chern connection $\nabla^{E,h}$ of $(E,h)$, $\Lambda_g$ is the dual Lefschetz operator of $(M,g)$, and $\lambda_{\min}(\cdot)$ denotes its smallest eigenvalue. Our first result is a surjectivity criterion for the real-linear Cauchy--Riemann operator $D$.

\begin{theorem}\label{thm:intro-L2}
If there exists a Hermitian metric $h$ on $E$ such that
	\begin{align}\label{eq:ioc}
	\int_M\theta_{E,h,g}\,dV_g+2\pi\chi(M)
	>\int_M\abs{A}_{g,h}^2\,dV_g,
	\end{align}
then $D$ is surjective.
\end{theorem}
The hypothesis \eqref{eq:ioc} is invariant under conformal changes of $g$ since both $\theta_{E,h,g}$ and $|A|_{g,h}^2$ scale inversely with the volume form in real dimension two. Although \eqref{eq:ioc} is formulated in terms of an integral inequality, a Poisson equation translates its strict positivity into a pointwise estimate for an appropriately weighted adjoint of $\ddbar_E$. The resulting estimate absorbs the real-linear zeroth-order term $A$ and yields the surjectivity of $D$.

The integral of the smallest eigenvalue can be intrinsically optimized over the space $\Herm(E)$ of smooth Hermitian metrics on $E$, in terms of the minimal Harder--Narasimhan slope (see \eqref{eq:mumin}).

\begin{theorem}\label{th:Op-integral}
For every holomorphic vector bundle $E\to M$,
	\begin{align*}
	\sup_{h\in\Herm(E)}\int_M\theta_{E,h,g}\,dV_g
	=2\pi\mu_{\min}^{\HN}(E).
	\end{align*}
\end{theorem}
For a line bundle, \Cref{th:Op-integral} reduces to the Chern--Weil formula. More generally, if $E$ is semistable, then $\mu_{\min}^{\HN}(E)=\mu(E)$, so the supremum of the integrated smallest eigenvalue is determined solely by $\mu(E)$. For a general bundle, the supremum in \Cref{th:Op-integral} need not be attained by a single Hermitian metric. It is nevertheless approached by a sequence of Harder--Narasimhan approximating metrics. This approximation principle, together with \Cref{thm:intro-L2}, leads naturally to the following criterion.
\begin{corollary}\label{cor:AHNc}
Let $\{h_\nu\}$ be a sequence of Hermitian metrics on $E$ such that
	\begin{align*}
	\int_M\theta_{E,h_\nu,g}\,dV_g
	\longrightarrow 2\pi\mu_{\min}^{\HN}(E).
	\end{align*}
If
	\begin{align*}
		2\pi\bigl(\mu_{\min}^{\HN}(E)+\chi(M)\bigr)>\limsup_{\nu\to\infty}\int_M\abs{A}_{g,h_\nu}^2\,dV_g,
	\end{align*}
then $D$ is surjective.
\end{corollary}

A useful lifting argument (\Cref{prop:filtered}) shows that, whenever $D$ preserves a holomorphic filtration, surjectivity of the operators induced on the successive quotients implies surjectivity of $D$. Combining this observation with the Harder--Narasimhan criterion \Cref{cor:AHNc} and the rank-one surjectivity theorem \cite[Theorem $1'$]{HLS} leads to the following quotientwise result.
\begin{theorem}\label{thm:intro-filtration}
Suppose $D=\ddbar_E+A$ preserves a holomorphic filtration
	\begin{align*}
		0=E_0\subset E_1\subset\cdots\subset E_\ell=E,
	\end{align*}
and set $Q_i:=E_i/E_{i-1}$. Assume that
	\begin{align*}
		\mu_{\min}^{\HN}(Q_i)>-\chi(M)
	\end{align*}
for every $1\leq i\leq\ell$. Suppose moreover that, for each $i$, one of the following conditions holds:
	\begin{enumerate}[label=\textup{(\roman*)}]
		\item $A$ strictly lowers the $i$-{\rm th} step of the filtration, i.e.,
		\begin{align*}
			A(E_i)\subseteq
			\Lambda^{0,1}T^*M\otimes E_{i-1};
		\end{align*}
		\item
$Q_i$ is a holomorphic line bundle.
	\end{enumerate}
Then $D$ is surjective.
\end{theorem}

The two alternatives correspond to complementary quotientwise conditions. In (i), the quotient $Q_i$ may have arbitrary rank, but the induced zero-order term on $Q_i$ vanishes. In (ii), the quotient has rank one, and no vanishing or size condition is imposed on the induced zero-order term. The two quotientwise conditions may be employed independently at different steps of the same filtration.

The first alternative is particularly natural for the Harder--Narasimhan filtration. If $A$ strictly lowers this filtration, then the induced zero-order term vanishes on every Harder--Narasimhan quotient. Since these quotients are semistable and their slopes are bounded below by $\mu_{\min}^{\HN}(E)$, \Cref{thm:intro-filtration} immediately gives the following metric-independent criterion.

\begin{corollary}\label{cor:ASNlc}
Assume that $A$ strictly lowers the Harder--Narasimhan filtration \eqref{eq:HNfiltration}, i.e.,
	\begin{align*}
		A(E_i)\subseteq
		\Lambda^{0,1}T^*M\otimes E_{i-1}\ \text{for every}\ 1\leq i\leq\ell,
	\end{align*}
and
	\begin{equation*}
		\mu_{\min}^{\HN}(E)>-\chi(M).
	\end{equation*}
Then $D$ is surjective.
\end{corollary}

We can apply the preceding operator-theoretic results to linearized Cauchy--Riemann operators on pullback tangent bundles and, for immersed curves, to the corresponding normal operators. These results yield higher-rank automatic transversality criteria for closed pseudoholomorphic curves under any of the integral, asymptotic Harder--Narasimhan, or quotientwise filtration conditions developed above (\Cref{cor:Jcurve}). We also obtain a relative automatic transversality criterion ensuring that Fredholm regularity is preserved under pseudoholomorphic immersions (\Cref{cor:relative-automatic}).

As a geometric application, \Cref{sec:S6} is devoted to the study of pseudoholomorphic spheres in $S^6$. Under certain geometric hypotheses, we prove an $L^2$ lower bound for the induced normal Nijenhuis term for arbitrary almost complex structures, together with an obstruction to the existence of invariant holomorphic line subbundles.

\begin{corollary}\label{cor:S6-symmetry'}
	Let $J$ be an almost complex structure on $S^6$, and let $u:\mathbb{CP}^1\to(S^6,J)$ be an immersed pseudoholomorphic sphere satisfying
	\begin{align}\label{eq:S6-splitting'}
		N_u\cong\mathcal O(-1)\oplus\mathcal O(-1).
	\end{align}
	Suppose that there exists $z_0\in\mathbb{CP}^1$ such that
	\begin{align*}
		T_{u(z_0)}\bigl(\operatorname{Aut}(S^6,J)\cdot u(z_0)\bigr)\not\subseteq {\rm d}u_{z_0}(T_{z_0}\mathbb{CP}^1).
	\end{align*}
	Let $h$ be a Hermitian--Einstein metric on $N_u$. Then
	\begin{enumerate}[label=\textup{(\roman*)}]
		\item The induced normal Nijenhuis term satisfies
		\begin{align*}
			\int_{\mathbb{CP}^1}
			|\mathcal N_{J,u}^{\mathrm N}|_{g,h}^2\,dV_g
			\geq32\pi.
		\end{align*}
		In particular, $J$ cannot be integrable.
		\item There is no holomorphic line subbundle $L\cong\mathcal O(-1)\subset N_u$ that is preserved by $\mathcal N_{J,u}^{\mathrm N}$.
	\end{enumerate}
\end{corollary}
For the standard nearly K\"ahler structure, we give a geometric characterization of the splitting hypothesis: \eqref{eq:S6-splitting'} holds if and only if $u$ is totally geodesic (\Cref{prop:S6-TG}). Since the $G_2$-action on $S^6$ is transitive, the transverse-orbit condition is then automatic. Consequently, the conclusions of \Cref{cor:S6-symmetry'} apply to every immersed totally geodesic pseudoholomorphic sphere (\Cref{cor:S6-TG}).

\section{Preliminaries}
\subsection{Real-linear Cauchy--Riemann operators and pseudoholomorphic curves}\label{sec:2.1}

Let $(M,j)$ be a closed Riemann surface, and let $E\to M$ be a holomorphic vector bundle. Set
\begin{align*}
	\Omega^{0,1}(E)
	=
	C^\infty(M,\Lambda^{0,1}T^*M\otimes E).
\end{align*}
A \emph{real-linear Cauchy--Riemann operator} on $E$ is a first-order real-linear differential operator
\begin{align*}
	D=\ddbar_E+A:
	C^\infty(M,E)
	\longrightarrow
	\Omega^{0,1}(E),
\end{align*}
where $\ddbar_E$ is the Cauchy--Riemann operator defining the holomorphic structure on $E$, and $A:E\longrightarrow\Lambda^{0,1}T^*M\otimes E$ is a real-linear bundle homomorphism. For $1<p<\infty$, the Sobolev extension
\begin{align*}
	D:
	W^{1,p}(M,E)
	\longrightarrow
	L^p(M,\Lambda^{0,1}T^*M\otimes E)
\end{align*}
is Fredholm. Its kernel and cokernel are represented by smooth sections by elliptic regularity, and its surjectivity is independent of $p$.

Let $g$ be a $j$-Hermitian metric on $M$ and $h$ a Hermitian metric on $E$. We denote by $\abs{A}_{g,h}(x)$ the operator norm of the fiber map
\begin{align*}
	A_x:
	(E_x,h_x)
	\longrightarrow
	\bigl(
	\Lambda^{0,1}T_{x}^*M\otimes E_x,
	g_x^*\otimes h_x
	\bigr).
\end{align*}
All $L^2$ inner products and formal adjoints below are taken with respect to these metrics. For real-linear operators, formal adjoints are understood with respect to the underlying real $L^2$ inner product. Let $\nabla^{E,h}$ be the Chern connection of $(E,\ddbar_E,h)$ and let $R^{E,h}$ denote its curvature. Define
\begin{align*}
	\theta_{E,h,g}(x)
	:=
	\lambda_{\min}
	\bigl(
	\ii\Lambda_gR^{E,h}(x)
	\bigr),
\end{align*}
where $\lambda_{\min}(\cdot)$ denotes the smallest eigenvalue and $\Lambda_g$ is the dual Lefschetz operator associated with $g$.

The following weighted $L^2$ estimate, derived from Ji--Zhu's half-Dirac estimate, plays an essential role in this paper.
\begin{proposition}[{\cite[Proposition 4.1 and Lemma 2.2]{JiZhu}}]
	Let $(E,\ddbar_E,h)$ be a Hermitian holomorphic vector bundle over a closed Riemann surface $(M,g)$. For every real-valued $\varphi\in C^2(M)$ and every $\eta\in\Omega^{0,1}(E)$,
	\begin{equation}\label{eq:weighted-Dolbeault}
		\int_M\bigl|\ddbar_{E,\varphi}^*\eta\bigr|^2e^{-\varphi}\,dV_g
		\geq
		\int_M\left(\frac12\Delta_g\varphi+\theta_{E,h,g}+K_g\right)|\eta|^2e^{-\varphi}\,dV_g,
	\end{equation}
	where $K_g$ is the Gaussian curvature, $\Delta_g$ is the Laplace--Beltrami operator, and $\ddbar_{E,\varphi}^*=e^{\varphi}\ddbar_E^*e^{-\varphi}$ denotes the formal adjoint of $\ddbar_E$ with respect to $e^{-\varphi}dV_g$.
\end{proposition}
Indeed, the estimate \eqref{eq:weighted-Dolbeault} is obtained by choosing $S^-=\Lambda^{0,1}T^*M\otimes E$ and $D^+=\sqrt{2}\,\ddbar_E$ in \cite[Proposition 4.1]{JiZhu}, together with the curvature term $R^-=\ii\Lambda_gR^{E,h}+K_g$ from \cite[Lemma 2.2]{JiZhu}. For broader applications of weighted $L^2$ methods in complex and algebraic geometry, we refer to Siu's foundational work on the invariance of plurigenera and the extension of twisted pluricanonical sections \cite{SiuPlurigenera,SiuPlurigenera2} and Guan--Zhou's resolution of the strong openness conjecture \cite{GuanZhou}.

We next recall how real-linear Cauchy--Riemann operators arise naturally in the theory of pseudoholomorphic curves. A smooth map
\begin{align*}
	f:(X',J')\longrightarrow(X,J)
\end{align*}
between almost complex manifolds is called \emph{pseudoholomorphic} if
\begin{align*}
	{\rm d}f\circ J'=J\circ{\rm d}f.
\end{align*}
In particular, if $(M,j)$ is a Riemann surface, a map
\begin{align*}
	u:(M,j)\longrightarrow(X,J)
\end{align*}
is pseudoholomorphic (or $J$-holomorphic) precisely when
\begin{align*}
	\ddbar_Ju:=\frac12\bigl({\rm d}u+J\circ{\rm d}u\circ j\bigr)=0.
\end{align*}
The linearized operator has the form
\begin{align*}
	D_u=\ddbar_u+A_u,
\end{align*}
where $\ddbar_u$ is a complex-linear Cauchy--Riemann operator on the vector bundle $(u^*TX,J)$ and $A_u$ is an anti-complex-linear bundle homomorphism. 

For later use, we give a connection description of the complex-linear and anti-complex-linear parts of the linearized Cauchy--Riemann operator $D_u$. Choose a connection $\nabla$ on $TX$ satisfying $\nabla J=0$, whose torsion $T$ is defined by $T(V,W):=\nabla_VW-\nabla_WV-[V,W]$ for $V,W\in C^\infty(X,TX)$. We use the standard convention 
\begin{align*} 
	N_J(V,W):=[JV,JW]-J[JV,W]-J[V,JW]-[V,W]. 
\end{align*} 
The Nijenhuis tensor satisfies 
\begin{align*} 
	N_J(JV,W)=N_J(V,JW)=-JN_J(V,W). 
\end{align*} 
Let $\xi\in C^\infty(M,u^*TX)$ and $Y\in C^\infty(M,TM)$. If $u_t$ is a variation of $u$ with variational vector field $\xi$ and $Y$ is extended independently of $t$, then the torsion identity gives 
\begin{align*} 
	\left.\nabla_{\partial_t}{\rm d}u_t(Y)\right|_{t=0}=\nabla_Y\xi+T(\xi,{\rm d}u(Y)). 
\end{align*} 
Since $\nabla J=0$ and ${\rm d}u(jY)=J{\rm d}u(Y)$, differentiating 
\begin{align*} 
	\ddbar_Ju_t(Y)=\frac12\bigl({\rm d}u_t(Y)+J{\rm d}u_t(jY)\bigr) 
\end{align*} 
at $t=0$ yields 
\begin{align}\label{eq:linearization} 
	D_u\xi(Y)=\frac12\bigl(\nabla_Y\xi+J\nabla_{jY}\xi\bigr)+B_u\xi(Y), 
\end{align} 
where 
\begin{align*} 
	B_u\xi(Y):=\frac12\bigl(T(\xi,{\rm d}u(Y))+JT(\xi,J{\rm d}u(Y))\bigr). 
\end{align*} 
The first term on the right-hand side of \eqref{eq:linearization} is complex-linear in $\xi$. Since $\nabla J=0$, the Nijenhuis tensor can be expressed in terms of the torsion as 
\begin{align*} 
	N_J(V,W)=T(V,W)+JT(JV,W)+JT(V,JW)-T(JV,JW). 
\end{align*} 
Hence the anti-complex-linear part of $B_u$ with respect to $\xi$ is 
\begin{align*} 
	B_u^{\mathrm{anti}}\xi(Y)=\frac12\bigl(B_u\xi(Y)+JB_u(J\xi)(Y)\bigr)= 
	\frac14N_J(\xi,{\rm d}u(Y)). 
\end{align*} 
It follows that 
\begin{align*} 
	A_u\xi(Y)=\frac14N_J\bigl(\xi,{\rm d}u(Y)\bigr). 
\end{align*} 
If $B_u^{\mathrm{lin}}$ denotes the complex-linear part of $B_u$ with respect to $\xi$, then 
\begin{align*} 
	\ddbar_u=\nabla^{0,1}+B_u^{\mathrm{lin}}. 
\end{align*} 
In particular, if $B_u$ is anti-complex-linear in $\xi$, then $B_u^{\mathrm{lin}}=0$ and 
\begin{align*} 
	\ddbar_u=\nabla^{0,1}, 
\end{align*} 
where $\nabla^{0,1}$ denotes the $(0,1)$-part of the induced connection on $u^*TX$ with respect to $j$. When $J$ is integrable, the Nijenhuis tensor vanishes, so $A_u=0$ and $D_u$ is complex-linear. In general, the surjectivity of $D_u$ implies that $u$ is Fredholm regular.

For an immersed pseudoholomorphic curve $u$, ${\rm d}u(TM)$ is a complex line subbundle of $u^*TX$, and the normal bundle
\begin{align*}
	N_u:=u^*TX/{\rm d}u(TM)
\end{align*}
is a smooth complex vector bundle. The linearization induces a normal real-linear Cauchy--Riemann operator
\begin{align*}
	D_u^{\mathrm N}:C^\infty(M,N_u)\longrightarrow\Omega^{0,1}(N_u).
\end{align*}
Let $\pi^{\mathrm N}:u^*TX\longrightarrow N_u$ denote the quotient map. The normal operator is characterized by
\begin{align}\label{eq:normal}
	D_u^{\mathrm N}\pi^{\mathrm N}=({\Id}\otimes\pi^{\mathrm N})D_u,
\end{align}
where $\Id$ denotes the identity on $\Lambda^{0,1}T^*M$. Taking the complex-linear part gives
\begin{align}\label{eq:normal-dbar}
	\ddbar_u^{\mathrm N}\pi^{\mathrm N}=({\Id}\otimes\pi^{\mathrm N})\ddbar_u.
\end{align}
Fredholm regularity is equivalent to the surjectivity of the normal Cauchy--Riemann operator $D_u^{\mathrm N}$ (see Wendl \cite[Proposition 2.2]{WendlSuperRigidity}). In particular, $N_u$ is a holomorphic line bundle if $\dim_\R X=4$, which is the key feature underlying the classical automatic transversality criterion.

\subsection{Harder--Narasimhan filtrations}
The degree of a holomorphic vector bundle $E$ is
\begin{equation*}
 \deg E
 :=\frac{1}{2\pi}\int_M
 \tr\bigl(\ii\Lambda_gR^{E,h}\bigr)\,dV_g.
\end{equation*}
The slope of $E$ is
\begin{align*}
 \mu(E):=\frac{\deg E}{\rank E}.
\end{align*}
A holomorphic vector bundle is \emph{semistable} if every nonzero proper holomorphic subbundle $S\subset E$ satisfies $\mu(S)\le\mu(E)$ (see \cite{HuybrechtsLehn}).

We first recall the following well-known result for holomorphic vector bundles over closed Riemann surfaces.
\begin{theorem}[\cite{HN}]
Every holomorphic vector bundle $E$ over a closed Riemann surface $M$ admits a unique filtration by holomorphic subbundles
	\begin{equation}\label{eq:HNfiltration}
		0=E_0\subset E_1\subset\cdots\subset E_\ell=E
	\end{equation}
whose quotients
	\begin{align*}
		Q_i:=E_i/E_{i-1}
	\end{align*}
are semistable and satisfy
	\begin{equation*}
		\mu(Q_1)>\mu(Q_2)>\cdots>\mu(Q_\ell).
	\end{equation*}
\end{theorem}
The filtration \eqref{eq:HNfiltration} is referred to as the \emph{Harder--Narasimhan filtration}. Set
\begin{equation}\label{eq:mumin}
 \mu_{\min}^{\HN}(E):=\mu(Q_\ell),
\end{equation}
and the minimal Harder--Narasimhan slope admits a standard quotient characterization as follows.

\begin{proposition}\label{prop:quotient-mumin}
For every holomorphic vector bundle $E$ over a closed Riemann surface,
\begin{equation*}
 \mu_{\min}^{\HN}(E)
 =\min_{E\twoheadrightarrow Q\ne0}\mu(Q),
\end{equation*}
where the minimum runs over all nonzero holomorphic quotient bundles $Q$ of $E$, and is attained by the final Harder--Narasimhan quotient in \eqref{eq:HNfiltration}.
\end{proposition}

\begin{proof}
This is a consequence of the Harder--Narasimhan filtration
(see, e.g., \cite[Example 3.20]{FuMu}); we include the proof for completeness. The final Harder--Narasimhan quotient
$Q_\ell=E/E_{\ell-1}$ is a nonzero holomorphic quotient of $E$ and satisfies $\mu(Q_\ell)=\mu_{\min}^{\HN}(E)$. It therefore remains to show that every nonzero holomorphic quotient $Q$ of $E$ satisfies $\mu(Q)\geq\mu_{\min}^{\HN}(E)$.

Now let $E\twoheadrightarrow Q$ be a nonzero quotient. Dualizing gives an injection $Q^*\hookrightarrow E^*$. Consider the Harder--Narasimhan filtration of $E^*$
	\begin{align*}
		0=H_0\subset H_1\subset\cdots\subset H_r=E^*.
	\end{align*}
Then each $H_i/H_{i-1}$ is semistable and
	\begin{align*}
		\mu_{\max}^{\HN}(E^*):=\mu(H_1)>\mu(H_2/H_{1})>\cdots>\mu(H_r/H_{r-1}).
	\end{align*}

The inclusion $Q^*\subset E^*$ induces a filtration by coherent subsheaves
	\begin{align*}
		0=S_0\subset S_1\subset\cdots\subset S_r=Q^*,
		\quad
		S_i:=Q^*\cap H_i.
	\end{align*}
For every $i$, the natural map
	\begin{align*}
		S_i/S_{i-1}\longrightarrow H_i/H_{i-1}
	\end{align*}
	is injective. Thus every nonzero $S_i/S_{i-1}$ is a torsion-free coherent subsheaf of the semistable bundle $H_i/H_{i-1}$. Let $\widetilde S_i$ denote its
	saturation in $H_i/H_{i-1}$. Then $\widetilde S_i\subseteq H_i/H_{i-1}$ is a holomorphic subbundle with the same rank as $S_i/S_{i-1}$, and
	\begin{align*}
		\deg(S_i/S_{i-1})\leq \deg \widetilde S_i.
	\end{align*}
	By semistability,
	\begin{align*}
		\mu(S_i/S_{i-1})\leq\mu(\widetilde S_i)\leq\mu(H_i/H_{i-1})\leq\mu_{\max}^{\HN}(E^*).
	\end{align*}
	Consequently,
	\begin{align*}
		\deg Q^*=\sum_i \deg(S_i/S_{i-1})\leq
		\sum_i
		\mu_{\max}^{\HN}(E^*)\rank(S_i/S_{i-1})=
		\mu_{\max}^{\HN}(E^*)\rank Q^*.
	\end{align*}
Dividing by $\rank Q^*$ gives
	\begin{align*}
		\mu(Q^*)
		\leq
		\mu_{\max}^{\HN}(E^*).
	\end{align*}

Finally, the Harder--Narasimhan filtration of the dual bundle reverses the order of the slopes and changes their signs. Thus,
	\begin{align*}
		\mu_{\max}^{\HN}(E^*)
		=
		-\mu_{\min}^{\HN}(E),\quad \mu(Q^*)=-\mu(Q),
	\end{align*}
and hence
	\begin{align*}
		\mu(Q)\geq\mu_{\min}^{\HN}(E).
	\end{align*}
Since $Q$ was arbitrary, every nonzero holomorphic quotient of $E$
has slope at least $\mu_{\min}^{\HN}(E)$. Since equality is attained
by the final Harder--Narasimhan quotient $Q_\ell$, the minimum exists
and is equal to $\mu_{\min}^{\HN}(E)$.
\end{proof}

\section{A conformal \texorpdfstring{$L^2$}{L2} surjectivity criterion}
In this section, we prove \Cref{thm:intro-L2} by incorporating the integrated curvature condition \eqref{eq:ioc} into the weighted estimate \eqref{eq:weighted-Dolbeault}, and illustrate the result with an example.

\begin{theorem}[=\Cref{thm:intro-L2}]\label{thm:L2criterion}
Let $D=\ddbar_E+A$ be a real-linear Cauchy--Riemann operator on a holomorphic vector bundle $E$ over a closed Riemann surface $M$. Suppose that for some Hermitian metric $h$ on $E$,
\begin{equation*}
 \int_M\theta_{E,h,g}\,dV_g+2\pi\chi(M)
 >\int_M\abs{A}_{g,h}^2\,dV_g.
\end{equation*}
Then $D$ is surjective.
\end{theorem}

\begin{proof}
By the Gauss--Bonnet theorem, the hypothesis implies
\begin{align*}
	\int_M\bigl(\theta_{E,h,g}+K_g-|A|_{g,h}^2\bigr)\,dV_g>0,
\end{align*}
where $K_g$ denotes the Gaussian curvature. Since $\theta_{E,h,g}+K_g-|A|_{g,h}^2$ is continuous, one may choose a smooth real-valued function $f$ such that
\begin{align*}
	f\leq\theta_{E,h,g}+K_g-|A|_{g,h}^2,
	\quad \int_M f\,dV_g>0.
\end{align*}
Set
\begin{align*}
	\delta:=\frac{1}{\Vol_g(M)}\int_M f\,dV_g>0.
\end{align*}
By the definition of $\delta$, the right-hand side of the Poisson equation
\begin{align*}
	\frac12\Delta_g\varphi=\delta-f
\end{align*}
has zero integral over $M$. Thus, the Poisson equation admits a smooth real-valued solution $\varphi$. Therefore, on $M$ we have
\begin{equation}\label{eq:weight-absorbs}
	\frac12\Delta_g\varphi+\theta_{E,h,g}+K_g
	\geq |A|_{g,h}^2+\delta.
\end{equation}
Combining \eqref{eq:weighted-Dolbeault} and \eqref{eq:weight-absorbs}, we obtain, for every $\eta\in\Omega^{0,1}(E)$,
\begin{equation}\label{eq:weighted-lower}
	\int_M|\ddbar_{E,\varphi}^*\eta|^2e^{-\varphi}\,dV_g
	\geq\int_M(|A|_{g,h}^2+\delta)|\eta|^2e^{-\varphi}\,dV_g.
\end{equation}

Since the weight is real-valued, the weighted adjoint of $D$ with respect to the weighted $L^2$ inner product is
\begin{align*}
	D_\varphi^*=\ddbar_{E,\varphi}^*+A^*.
\end{align*}
Let $\eta\in\ker D_\varphi^*$, which is smooth by elliptic regularity. Then
\begin{align*}
	\ddbar_{E,\varphi}^*\eta=-A^*\eta.
\end{align*}
It follows from \eqref{eq:weighted-lower} and $|A^*\eta|\leq|A^*|_{g,h}|\eta|=|A|_{g,h}|\eta|$ that
\begin{align*}
	\int_M(|A|_{g,h}^2+\delta)|\eta|^2e^{-\varphi}\,dV_g
	\leq\int_M|A^*\eta|^2e^{-\varphi}\,dV_g
	\leq\int_M|A|_{g,h}^2|\eta|^2e^{-\varphi}\,dV_g.
\end{align*}
Consequently,
\begin{align*}
	\delta\int_M|\eta|^2e^{-\varphi}\,dV_g\leq0,
\end{align*}
which gives $\eta=0$. It follows that $\ker D_\varphi^*=0$. As $D$ is an elliptic Fredholm operator on the closed surface,
\begin{align*}
 \coker D\cong\ker D_\varphi^*=0.
\end{align*}
Hence $D$ is surjective.
\end{proof}

\begin{example}\label{ex:split-curvature}
	Let $M=\mathbb{CP}^1$ equipped with the Fubini--Study metric, and let
	\begin{align*}
		E=\mathcal O(d_1)\oplus\cdots\oplus\mathcal O(d_\ell),
		\quad
		d_1\leq d_2\leq\cdots\leq d_\ell.
	\end{align*}
	Equip each summand $\mathcal O(d_i)$ with its Hermitian--Einstein
	metric $h_i$ and $E$ with the orthogonal direct sum metric $h$. Hence the curvature endomorphism of $E$ is diagonal, and $\theta_{E,h,g}={\ii}\Lambda_g R^{\mathcal O(d_1),h_1}=2\pi d_1/\Vol_g(\mathbb{CP}^1)$. Therefore,
	\begin{align*}
		\int_M\theta_{E,h,g}\,dV_g=2\pi d_1.
	\end{align*}
	Since $\chi(\mathbb{CP}^1)=2$, \Cref{thm:L2criterion} implies that every real-linear Cauchy--Riemann operator
	\begin{align*}
		D=\ddbar_E+A
	\end{align*}
is surjective whenever
	\begin{align*}
		\int_M |A|_{g,h}^2\,dV_g
		<
		2\pi(d_1+2).
	\end{align*}
Thus the surjectivity criterion depends only on the lowest curvature direction of the split bundle, while the zero-order term $A$ may couple the different summands arbitrarily, provided its $L^2$-energy satisfies the above bound.
\end{example}

\section{Harder--Narasimhan optimization of the integrated smallest eigenvalue}
This section is devoted to the proofs of \Cref{th:Op-integral} and \Cref{cor:AHNc}. We also provide a classical cohomological interpretation of \Cref{cor:AHNc} in the case where the real-linear term $A=0$.

We begin with the following upper bound for the integrated curvature term in terms of the minimal Harder--Narasimhan slope.
\begin{lemma}\label{lem:HN-upper}
	For every Hermitian metric $h$ on $E$,
	\begin{equation}\label{eq:HN-upper}
		\int_M\theta_{E,h,g}\,dV_g
		\leq 2\pi\mu_{\min}^{\HN}(E).
	\end{equation}
\end{lemma}

\begin{proof}
	Let $\pi:E\twoheadrightarrow Q$ be any nonzero holomorphic quotient bundle, and let $h_Q$ be the quotient metric induced by $h$. Identify $Q$
	smoothly with the orthogonal complement $(\ker\pi)^\perp$. With respect to the orthogonal splitting
	$E=\ker\pi\oplus (\ker\pi)^\perp$, the standard quotient curvature formula gives
	\begin{align*}
		R^{Q,h_Q}=P_QR^{E,h}|_Q+\beta\wedge\beta^*,
	\end{align*}
	where $P_Q:E\to Q\simeq(\ker\pi)^\perp$ denotes the orthogonal projection and $\beta$ is the second fundamental form. With our curvature
	convention, ${\ii}\Lambda_g(\beta\wedge\beta^*)$ is positive
	semidefinite. Since, by definition,
	\begin{align*}
		{\ii}\Lambda_gR^{E,h}
		\geq
		\theta_{E,h,g}\Id_E,
	\end{align*}
	it follows that
	\begin{align*}
		{\ii}\Lambda_gR^{Q,h_Q}
		\geq
		\theta_{E,h,g}\Id_Q.
	\end{align*}
	Taking the trace and integrating, we obtain
	\begin{align*}
		2\pi\deg Q=\int_M\tr\bigl({\ii}\Lambda_gR^{Q,h_Q}\bigr)\,dV_g\geq
		\rank Q\int_M\theta_{E,h,g}\,dV_g.
	\end{align*}
	Hence
	\begin{align*}
		\int_M\theta_{E,h,g}\,dV_g
		\leq
		2\pi\mu(Q).
	\end{align*}
	Since $Q$ was arbitrary, \Cref{prop:quotient-mumin} yields
	\eqref{eq:HN-upper}.
\end{proof}

We now recall the following consequence of Bradlow's Hermitian--Einstein inequalities \cite[Theorem 5]{Bradlow}, restated using our conventions for curvature and volume.

\begin{theorem}\label{thm:Bradlow}
Let $E$ be a holomorphic vector bundle over a closed Riemann surface $(M,g)$. For every $\varepsilon>0$, there exists a smooth Hermitian metric $h_\varepsilon$ on $E$ such that
\begin{equation}\label{eq:Bradlow-lower}
 \ii\Lambda_gR^{E,h_\varepsilon}
 \ge
 \left(
 \frac{2\pi\mu_{\min}^{\HN}(E)}{\Vol_g(M)}-\varepsilon
 \right)\Id_E.
\end{equation}
If the Harder--Narasimhan filtration splits holomorphically and every Harder--Narasimhan quotient is polystable, one may take $\varepsilon=0$ with the orthogonal direct sum of Hermitian--Einstein metrics on the quotients.
\end{theorem}

\begin{theorem}[=\Cref{th:Op-integral}]\label{thm:optimal-curvature}
For every holomorphic vector bundle $E$ over a closed Riemann surface $M$,
\begin{equation*}
 \sup_{h\in\Herm(E)}
 \int_M\theta_{E,h,g}\,dV_g
 =2\pi\mu_{\min}^{\HN}(E).
\end{equation*}
\end{theorem}

\begin{proof}
The upper bound is given by \Cref{lem:HN-upper}. For the reverse inequality, choose the metric $h_\varepsilon$ from \Cref{thm:Bradlow}. Taking the smallest eigenvalue in \eqref{eq:Bradlow-lower} and integrating implies
\begin{align*}
 \int_M\theta_{E,h_\varepsilon,g}\,dV_g
 \ge
 2\pi\mu_{\min}^{\HN}(E)-\varepsilon\Vol_g(M).
\end{align*}
Taking $\varepsilon\searrow 0$, we get the lower bound.
\end{proof}
\begin{remark}
	(i) If $E$ is a line bundle, the unique quotient of its Harder--Narasimhan filtration is $E$ and $\mu_{\min}^{\HN}(E)=\deg(E)$. Hence \Cref{thm:optimal-curvature} reduces to the Chern--Weil formula.
	
	(ii) If the Harder--Narasimhan filtration of $E$ splits holomorphically and its graded factors are polystable, then each factor admits a Hermitian--Einstein metric (see \cite[Theorem 5]{Bradlow}). Taking the orthogonal direct sum of these metrics produces a Hermitian metric on $E$ for which the supremum in \Cref{thm:optimal-curvature} is attained. In particular, this applies when $E$ itself is polystable.
	
	(iii) If $E$ is semistable, its Harder--Narasimhan filtration has a single slope, namely $\mu_{\min}^{\HN}(E)=\mu(E)$. Thus the supremum of the integrated smallest eigenvalue is determined solely by $\mu(E)$.
\end{remark}

\Cref{thm:optimal-curvature} converts \Cref{thm:L2criterion} into a Harder--Narasimhan criterion, but the same Hermitian metric must be used to measure both the curvature term and $A$.

\begin{corollary}[=\Cref{cor:AHNc}]\label{cor:HN-A}
Let $h_\nu$ be a sequence of Hermitian metrics such that
\begin{equation}\label{eq:HN-approx}
 \int_M\theta_{E,h_\nu,g}\,dV_g
 \longrightarrow2\pi\mu_{\min}^{\HN}(E).
\end{equation}
If
\begin{equation}\label{eq:HN-limsup}
 \limsup_{\nu\to\infty}
 \int_M\abs{A}_{g,h_\nu}^2\,dV_g
 <2\pi\bigl(\mu_{\min}^{\HN}(E)+\chi(M)\bigr),
\end{equation}
then $D$ is surjective.
\end{corollary}

\begin{proof}
Set
\begin{align*}
 C_\nu:=\int_M\theta_{E,h_\nu,g}\,dV_g,
 \quad
 B_\nu:=\int_M\abs{A}_{g,h_\nu}^2\,dV_g.
\end{align*}
By \eqref{eq:HN-approx},
\begin{align*}
 C_\nu+2\pi\chi(M)
 \longrightarrow
 2\pi\bigl(\mu_{\min}^{\HN}(E)+\chi(M)\bigr).
\end{align*}
It follows from \eqref{eq:HN-limsup} that
\begin{align*}
 \liminf_{\nu\to\infty}
 \bigl(C_\nu+2\pi\chi(M)-B_\nu\bigr)>0.
\end{align*}
Thus the hypothesis of \Cref{thm:L2criterion} holds for all sufficiently large $\nu$. Applying \Cref{thm:L2criterion} for such $\nu$ proves that $D$ is surjective.
\end{proof}

\begin{remark}\label{rmk:cohomological-interpretation}
	When $A=0$, the slope condition
	\begin{align*}
		\mu_{\min}^{\HN}(E)>-\chi(M)
	\end{align*}
	has a classical cohomological interpretation. Let
	\begin{align*}
		0=E_0\subset E_1\subset\cdots\subset E_\ell=E
	\end{align*}
	be the Harder--Narasimhan filtration of $E$, and write $Q_i:=E_i/E_{i-1}$. Since each $Q_i$ is semistable, for every $i$ we have
	\begin{align*}
		\mu(Q_i)\geq\mu_{\min}^{\HN}(E)>-\chi(M)=\deg K_M.
	\end{align*}
	
	We first show that
	\begin{align*}
		H^1(M,Q_i)=0
	\end{align*}
	for every Harder--Narasimhan quotient $Q_i$. By Serre duality,
	\begin{align*}
		H^1(M,Q_i)^*\cong\operatorname{Hom}_{\mathcal O_M}(Q_i,K_M).
	\end{align*}
	Suppose, to the contrary, that there exists a nonzero holomorphic morphism
	\begin{align*}
		\varphi:Q_i\longrightarrow K_M.
	\end{align*}
	Since $Q_i$ and $K_M$ are locally free coherent sheaves, $\operatorname{Im}\varphi$ is a coherent subsheaf of $K_M$. Moreover, $\operatorname{Im}\varphi$ is nonzero and torsion-free, because it is a subsheaf of the line bundle $K_M$. Since $M$ is a smooth complex curve, every torsion-free coherent sheaf is locally free. Thus $\operatorname{Im}\varphi$ is a holomorphic line bundle.
	
	The morphism $\varphi$ gives an exact sequence of vector bundles
	\begin{align}\label{eq:Serre-quotient}
		0\longrightarrow\ker\varphi\longrightarrow Q_i\longrightarrow \operatorname{Im}\varphi\longrightarrow0.
	\end{align}
	Since $Q_i$ is semistable, every nonzero holomorphic quotient bundle of $Q_i$ has slope at least $\mu(Q_i)$ by \Cref{prop:quotient-mumin}.
	Hence \eqref{eq:Serre-quotient} gives
	\begin{align}\label{eq:image-lower}
		\mu(\operatorname{Im}\varphi)\geq\mu(Q_i).
	\end{align}
	
	On the other hand, the inclusion $\operatorname{Im}\varphi\subset K_M$ yields an exact sequence
	\begin{align}\label{eq:image-canonical}
		0\longrightarrow \operatorname{Im}\varphi\longrightarrow K_M\longrightarrow T\longrightarrow0,
	\end{align}
	where $T$ is a torsion sheaf. Taking degrees in \eqref{eq:image-canonical}, we obtain
	\begin{align*}
		\deg K_M=\deg \operatorname{Im}\varphi+\operatorname{length}(T),
	\end{align*}
	and therefore
	\begin{align}\label{eq:image-upper}
		\mu(\operatorname{Im}\varphi)=\deg \operatorname{Im}\varphi\leq\deg K_M=-\chi(M).
	\end{align}
	Combining \eqref{eq:image-lower} and \eqref{eq:image-upper} gives
	\begin{align*}
		\mu(Q_i)\leq\mu(\operatorname{Im}\varphi)\leq-\chi(M),
	\end{align*}
	which contradicts
	\begin{align*}
		\mu(Q_i)>-\chi(M).
	\end{align*}
	Consequently,
	\begin{align*}
		\operatorname{Hom}_{\mathcal O_M}(Q_i,K_M)=0,
	\end{align*}
	and Serre duality gives
	\begin{align*}
		H^1(M,Q_i)=0
	\end{align*}
	for every $i$.
	
	It remains to pass from the Harder--Narasimhan quotients to $E$. For each $i$, the short exact sequence
	\begin{align*}
		0\longrightarrow E_{i-1}\longrightarrow E_i\longrightarrow Q_i\longrightarrow0
	\end{align*}
	induces the exact cohomology sequence
	\begin{align*}
		H^1(M,E_{i-1})\longrightarrow H^1(M,E_i)\longrightarrow H^1(M,Q_i)\longrightarrow H^2(M,E_{i-1}).
	\end{align*}
	Since $M$ is a Riemann surface,
	\begin{align*}
		H^2(M,E_{i-1})=0.
	\end{align*}
	Starting from $E_0=0$ and using $H^1(M,Q_i)=0$ for every $i$, an induction on $i$ therefore gives
	\begin{align*}
		H^1(M,E_i)=0
	\end{align*}
	for all $i$, and in particular
	\begin{align*}
		H^1(M,E)=0.
	\end{align*}
	By the Dolbeault theorem, this is equivalent to the surjectivity of
	\begin{align*}
		\ddbar_E:C^\infty(M,E)\longrightarrow\Omega^{0,1}(E).
	\end{align*}
	
	Thus, in the complex-linear case $A=0$, the slope condition $\mu_{\min}^{\HN}(E)>-\chi(M)$ yields surjectivity by the classical combination of Serre duality and Harder--Narasimhan theory. The real-linear result above extends this conclusion to operators $D=\ddbar_E+A$ by controlling the zero-order term through its $L^2$-energy along Harder--Narasimhan approximating metrics.
\end{remark}

We end this section by providing an example to illustrate \Cref{cor:HN-A}.
\begin{example}\label{ex:HN-extension}
	Let $M$ be a closed Riemann surface of genus $\gamma\geq2$. Choose integers
	\begin{align*}
		d_->2\gamma-2,
		\quad
		1\leq\delta\leq2\gamma-2,
	\end{align*}
	and set $d_+:=d_-+\delta$. Equivalently, one may write
	\begin{align*}
		d_-=2\gamma-2+q,
		\quad
		d_+=2\gamma-2+q+\delta,
		\quad
		q\geq1.
	\end{align*}
	
	Choose an effective divisor $Z$ of degree
	\begin{align*}
		\deg Z=2\gamma-2-\delta,
	\end{align*}
	and then $\deg K_M(-Z)=\delta$. Moreover, by Serre duality,
	\begin{align*}
		H^1(M,K_M(-Z))
		\cong
		H^0(M,K_M\otimes (K_M(-Z))^{-1})^*=
		H^0(M,\mathcal O(Z))^*
		\neq0,
	\end{align*}
	where the last inequality follows from the effectiveness of $Z$.
	
	Select a holomorphic line bundle $L_-$ of degree $d_-$ and define
	\begin{align*}
		L_+:=L_-\otimes K_M(-Z).
	\end{align*}
	Then $\deg L_+=d_+$ and $L_+\otimes L_-^{-1}\cong K_M(-Z)$. Set
	\begin{align*}
		Q_+:=L_+^{\oplus2},
		\quad
		Q_-:=L_-^{\oplus2}.
	\end{align*}
	Since
	\begin{align*}
		\operatorname{Ext}^1(Q_-,Q_+)
		\cong
		H^1\bigl(M,\operatorname{Hom}(Q_-,Q_+)\bigr)\cong
		H^1(M,K_M(-Z))\otimes\operatorname{End}(\mathbb C^2)
		\neq0,
	\end{align*}
	we may choose a nonsplit extension
	\begin{align*}
		0\longrightarrow Q_+
		\longrightarrow E
		\longrightarrow Q_-
		\longrightarrow0.
	\end{align*}
	The bundles $Q_+$ and $Q_-$ are polystable of slopes $d_+$ and $d_-$, respectively, and $d_+>d_-$. Hence
	\begin{align*}
		0\subset Q_+\subset E
	\end{align*}
	is the Harder--Narasimhan filtration of $E$, and
	\begin{align*}
		\mu_{\min}^{\HN}(E)=d_-.
	\end{align*}
	
	Choose Hermitian--Einstein metrics $h_+$ and $h_-$ on $Q_+$ and $Q_-$ and a smooth splitting
	\begin{align*}
		E\cong Q_+\oplus Q_-.
	\end{align*}
	With respect to this splitting, the holomorphic structure on $E$ has the form
	\begin{align*}
		\ddbar_E
		=
		\begin{pmatrix}
			\ddbar_{Q_+}&\beta\\
			0&\ddbar_{Q_-}
		\end{pmatrix},
	\end{align*}
	where $\beta\in\Omega^{0,1}(\operatorname{Hom}(Q_-,Q_+))$ represents the nonzero extension class. Consider the family of Hermitian metrics
	\begin{align*}
		h_t=t^2h_+\oplus h_-,
		\quad
		t\longrightarrow0.
	\end{align*}
	Since $h_t$ restricts to $t^2h_+$ on $Q_+$ and to the fixed metric $h_-$ on $Q_-$, while $\beta$ is a $(0,1)$-form with values in $\operatorname{Hom}(Q_-,Q_+)$, its pointwise norm satisfies
	\begin{align*}
		|\beta|_{g,h_t}=t\,|\beta|_{g,h_+\oplus h_-}.
	\end{align*}
	Thus the extension term becomes asymptotically negligible, and a direct computation shows that
	\begin{align*}
		\int_M\theta_{E,h_t,g}\,dV_g
		\longrightarrow
		2\pi\mu_{\min}^{\HN}(E)
		=
		2\pi d_-.
	\end{align*}
	
	We select a nonzero smooth real-linear bundle homomorphism
	\begin{align*}
		\Phi:Q_+\longrightarrow\Lambda^{0,1}T^*M\otimes Q_+,
	\end{align*}
	and define
	\begin{align*}
		A=
		\begin{pmatrix}
			\Phi&0\\
			0&0
		\end{pmatrix},
		\quad
		D=\ddbar_E+A.
	\end{align*}
	Since $\Phi$ is a $(0,1)$-form with values in $\operatorname{End}_{\mathbb R}(Q_+)$, a constant rescaling of $h_+$ does not change its pointwise operator norm. Hence
	\begin{align*}
		\int_M|A|_{g,h_t}^2\,dV_g
		=
		\int_M|A|_{g,h_+\oplus h_-}^2\,dV_g
	\end{align*}
	for every $t>0$.
	
	Since $\chi(M)=2-2\gamma$, we have
	\begin{align*}
		2\pi\bigl(\mu_{\min}^{\HN}(E)+\chi(M)\bigr)=
		2\pi\bigl(d_-+2-2\gamma\bigr)=2\pi q,
		\quad
		q\geq1.
	\end{align*}
	Therefore \Cref{cor:HN-A} gives
	\begin{align*}
		\int_M|A|_{g,h_+\oplus h_-}^2\,dV_g
		<
		2\pi q
		\quad\Longrightarrow\quad
		D\text{ is surjective}.
	\end{align*}
\end{example}

\section{Surjectivity via holomorphic filtrations}
In this section, we establish quotientwise surjectivity criteria for real-linear Cauchy--Riemann operators preserving holomorphic filtrations. We also give several examples that exhibit the different filtration mechanisms in these criteria.

\begin{definition}
A real-linear Cauchy--Riemann operator $D=\ddbar_E+A$ is said to \textit{preserve a holomorphic filtration}
\begin{equation}\label{eq:general-filtration}
 0=E_0\subset E_1\subset\cdots\subset E_m=E
\end{equation}
if for each $1\le i\le m$
\begin{equation}\label{eq:D-invariant}
 D:C^\infty(M,E_i)\longrightarrow\Omega^{0,1}(E_i).
\end{equation}
\end{definition}
Since every $E_i$ in \eqref{eq:general-filtration} is a holomorphic subbundle, condition \eqref{eq:D-invariant} is equivalent to
\begin{align*}
	A(E_i)\subseteq\Lambda^{0,1}T^*M\otimes E_i.
\end{align*}
Let $Q_i:=E_i/E_{i-1}$ and $p_i:E_i\to Q_i$ be the quotient map. If $D$ preserves the filtration, for $s\in C^\infty(M,E_i)$ we define
\begin{align}\label{eq:induced-D}
 D_i:C^\infty(M,Q_i)&\longrightarrow\Omega^{0,1}(Q_i)\nonumber\\
 p_i(s)&\longmapsto p_i(Ds),
\end{align}
where by abuse of notation, $p_i$ also denotes the induced map on spaces of $(0,1)$-forms.
\begin{lemma}
The operator $D_i$ in \eqref{eq:induced-D} is well-defined and
\begin{align*}
 D_i=\ddbar_{Q_i}+A_i,
\end{align*}
where $\ddbar_{Q_i}$ and $A_i$ are induced by $\ddbar_E$ and $A$, respectively.
\end{lemma}

\begin{proof}
If $s,s'\in C^\infty(M,E_i)$ are two lifts of the same section of $Q_i$, then $s'-s\in C^\infty(M,E_{i-1})$. By invariance,
\begin{align*}
 D(s'-s)\in\Omega^{0,1}(E_{i-1}),
\end{align*}
so $p_i(Ds')=p_i(Ds)$. Thus \eqref{eq:induced-D} is well-defined. Since $E_i$ and $E_{i-1}$ are holomorphic subbundles, $\ddbar_E$ induces the quotient holomorphic structure $\ddbar_{Q_i}$ on $Q_i$, and the zero-order map $A$ induces $A_i$.
\end{proof}

\begin{proposition}\label{prop:filtered}
Assume that $D$ preserves the holomorphic filtration \eqref{eq:general-filtration}. If every induced operator
\begin{align*}
 D_i:C^\infty(M,Q_i)\longrightarrow\Omega^{0,1}(Q_i)
\end{align*}
is surjective, then $D$ is surjective.
\end{proposition}

\begin{proof}
The argument is similar to that in \Cref{rmk:cohomological-interpretation}. We argue by induction on $m$. The assertion is immediate for $m=1$. Assume it has been proved for filtrations of length $m-1$. Consider the smooth short exact sequence
\begin{equation*}
 0\longrightarrow E_{m-1}\longrightarrow E
 \stackrel{p_m}{\longrightarrow}Q_m\longrightarrow0.
\end{equation*}
Every short exact sequence of smooth vector bundles splits after choosing a Hermitian metric, so every smooth section of $Q_m$ has a smooth lift to $E$.

Let $\eta\in\Omega^{0,1}(E)$. Since $D_m$ is surjective, there exists $\bar s\in C^\infty(M,Q_m)$ satisfying
\begin{align*}
 D_m\bar s=p_m(\eta).
\end{align*}
Choose a smooth lift $s\in C^\infty(M,E)$ with $p_m(s)=\bar s$. Compatibility of $D$ with the quotient gives
\begin{align*}
 p_m(\eta-Ds)=p_m(\eta)-D_m\bar s=0.
\end{align*}
Therefore
\begin{align*}
 \eta-Ds\in\Omega^{0,1}(E_{m-1}).
\end{align*}
The restriction of $D$ to $E_{m-1}$ preserves the shortened filtration
\begin{align*}
 0=E_0\subset\cdots\subset E_{m-1},
\end{align*}
and its quotient operators are $D_1,\ldots,D_{m-1}$. By the induction hypothesis, there exists $t\in C^\infty(M,E_{m-1})$ such that
\begin{align*}
 Dt=\eta-Ds.
\end{align*}
Then $D(s+t)=\eta$. Since $\eta$ was arbitrary, $D$ is surjective.
\end{proof}

\Cref{prop:filtered} reduces the surjectivity problem to the successive quotient operators. We first combine it with the Harder--Narasimhan criterion \Cref{cor:HN-A}.

\begin{corollary}
Assume that $D$ preserves the holomorphic filtration \eqref{eq:general-filtration}. Let $D_i=\ddbar_{Q_i}+A_i$ be the induced operator on the quotient $Q_i$. For every $i$, suppose that there exists a sequence of Hermitian metrics $h_{i,\nu}$ on $Q_i$ such that
\begin{equation}\label{eq:quotient-approx}
 \int_M\theta_{Q_i,h_{i,\nu},g}\,dV_g
 \longrightarrow 2\pi\mu_{\min}^{\HN}(Q_i)
\end{equation}
and
\begin{equation}\label{eq:quotient-energy}
 \limsup_{\nu\to\infty}
 \int_M\abs{A_i}_{g,h_{i,\nu}}^2\,dV_g
 <2\pi\bigl(\mu_{\min}^{\HN}(Q_i)+\chi(M)\bigr).
\end{equation}
Then $D$ is surjective.
\end{corollary}

\begin{proof}
Fix $i$. Applying \Cref{cor:HN-A} to the real-linear Cauchy--Riemann operator $D_i=\ddbar_{Q_i}+A_i$ with the sequence $h_{i,\nu}$, we conclude from \eqref{eq:quotient-energy} that $D_i$ is surjective. The conclusion therefore follows from \Cref{prop:filtered}.
\end{proof}

The criterion still requires quantitative control of the induced zero-order terms. Such control is, however, unnecessary in two important cases: either the induced zero-order term on a quotient vanishes, or the quotient has rank one. The rank-one case is closely related to the classical similarity principle for elliptic Cauchy--Riemann equations, which has also been extended to certain non-elliptic complex vector fields in the plane (see Berhanu--Hounie--Santiago \cite{BHS}). Here we use the following surjectivity theorem of Hofer--Lizan--Sikorav.
\begin{theorem}[{\cite[Theorem 1']{HLS}}]\label{thm:HLSline}
	Let $L\to M$ be a holomorphic line bundle over a closed Riemann surface, and let
	\begin{align*}
		D:C^\infty(M,L)\longrightarrow\Omega^{0,1}(L)
	\end{align*}
	be an arbitrary real-linear Cauchy--Riemann operator. If
	\begin{equation*}
		\deg L>-\chi(M),
	\end{equation*}
	then $D$ is surjective.
\end{theorem}

\begin{theorem}[=\Cref{thm:intro-filtration}]\label{thm:filtration-criteria}
Suppose $D=\ddbar_E+A$ preserves a holomorphic filtration
	\begin{align*}
		0=E_0\subset E_1\subset\cdots\subset E_\ell=E,
	\end{align*}
and set $Q_i:=E_i/E_{i-1}$. Assume that
	\begin{align}\label{eq:slope}
		\mu_{\min}^{\HN}(Q_i)>-\chi(M)
	\end{align}
for every $1\leq i\leq\ell$. Suppose moreover that, for each $i$, one of the following conditions holds:
	\begin{enumerate}[label=\textup{(\roman*)}]
		\item $A$ strictly lowers the $i$-{\rm th} step of the filtration, i.e.,
		\begin{align*}
			A(E_i)\subseteq
			\Lambda^{0,1}T^*M\otimes E_{i-1};
		\end{align*}
		\item
$Q_i$ is a holomorphic line bundle.
	\end{enumerate}
Then $D$ is surjective.
\end{theorem}

\begin{proof}
Fix $i$. If $A$ strictly lowers the $i$-{\rm th} step of the filtration, then the induced zero-order term $A_i=0$ on $Q_i$. Choose a sequence $h_{i,\nu}$ satisfying \eqref{eq:quotient-approx}. Then
\begin{align*}
 \int_M\abs{A_i}_{g,h_{i,\nu}}^2\,dV_g=0
 \ \text{for every }\nu.
\end{align*}
The slope hypothesis \eqref{eq:slope} is equivalent to
\begin{align*}
 2\pi\bigl(\mu_{\min}^{\HN}(Q_i)+\chi(M)\bigr)>0,
\end{align*}
so \eqref{eq:quotient-energy} holds. Hence \Cref{cor:HN-A} yields the surjectivity of the induced operator $D_i$.

If $Q_i$ is a holomorphic line bundle, the slope hypothesis \eqref{eq:slope} means that
\begin{align*}
	\deg Q_i=\mu_{\min}^{\HN}(Q_i)>-\chi(M).
\end{align*}
It follows from \Cref{thm:HLSline} that $D_i$ is surjective.

The filtered surjectivity principle \Cref{prop:filtered} then gives the surjectivity of $D$.
\end{proof}

\begin{example}\label{ex:line-flag}
Let $M$ be a closed Riemann surface of genus $\gamma$. Choose holomorphic line bundles $L_1,L_2,L_3$ of degrees
\begin{align*}
	\deg L_i>2\gamma-2,
	\quad
	1\leq i\leq3.
\end{align*}
Set
\begin{align*}
	E=L_1\oplus L_2\oplus L_3,
	\quad
	F_k=L_1\oplus\cdots\oplus L_k.
\end{align*}
Then
\begin{align*}
	0\subset F_1\subset F_2\subset F_3=E
\end{align*}
is a holomorphic filtration whose successive quotients are
\begin{align*}
	Q_1\cong L_1,
	\quad
	Q_2\cong L_2,
	\quad
	Q_3\cong L_3.
\end{align*}

For $i\leq j$, let
\begin{align*}
	A_{ij}:L_j\longrightarrow\Lambda^{0,1}T^*M\otimes L_i
\end{align*}
be arbitrary smooth real-linear bundle homomorphisms. Define
\begin{align*}
	D=
	\begin{pmatrix}
		\ddbar_{L_1}+A_{11}&A_{12}&A_{13}\\
		0&\ddbar_{L_2}+A_{22}&A_{23}\\
		0&0&\ddbar_{L_3}+A_{33}
	\end{pmatrix}.
\end{align*}
Then $D$ preserves the holomorphic filtration
\begin{align*}
	0\subset F_1\subset F_2\subset F_3=E,
\end{align*}
and the induced operators on the successive quotients are
\begin{align*}
	D_i=\ddbar_{L_i}+A_{ii},
	\quad
	1\leq i\leq3.
\end{align*}
Since each $Q_i\cong L_i$ is a holomorphic line bundle and
\begin{align*}
	\mu_{\min}^{\HN}(Q_i)=\deg L_i>2\gamma-2=-\chi(M),
\end{align*}
\Cref{thm:filtration-criteria} implies that $D$ is surjective.
\end{example}

\begin{corollary}[=\Cref{cor:ASNlc}]\label{cor:strict-HN}
Assume that $A$ strictly lowers the Harder--Narasimhan filtration \eqref{eq:HNfiltration}, i.e.,
\begin{align}\label{eq:strict-HN}
	A(E_i)\subseteq
	\Lambda^{0,1}T^*M\otimes E_{i-1}\ \text{for every}\ 1\leq i\leq\ell,
\end{align}
and
\begin{equation*}
 \mu_{\min}^{\HN}(E)>-\chi(M).
\end{equation*}
Then $D$ is surjective.
\end{corollary}

\begin{proof}
Condition \eqref{eq:strict-HN} implies that the induced real-linear term $A_i=0$. Moreover,
\begin{align*}
 \mu_{\min}^{\HN}(Q_i)=\mu(Q_i)\ge\mu_{\min}^{\HN}(E)>-\chi(M)
 \ \text{for}\ 1\le i\le\ell.
\end{align*}
Thus all hypotheses of \Cref{thm:filtration-criteria} are satisfied, and $D$ is surjective.
\end{proof}

\begin{example}\label{ex:HN-flag}
	Let $M$ be a closed Riemann surface of genus $\gamma$. Choose semistable holomorphic vector bundles $Q_+$ and $Q_-$ of ranks $r_+$ and $r_-$ and degrees $d_+$ and $d_-$, respectively, such that
	\begin{align*}
		\frac{d_+}{r_+}>\frac{d_-}{r_-}>2\gamma-2.
	\end{align*}
	Set $E=Q_+\oplus Q_-$. Then
	\begin{align*}
		0\subset Q_+\subset E,
	\end{align*}
	is the Harder--Narasimhan filtration of $E$ with successive quotients $Q_+$ and $Q_-$, since $Q_+$ and $Q_-$ are semistable and $\mu(Q_+)>\mu(Q_-)$. In particular,
	\begin{align*}
		\mu_{\min}^{\HN}(E)=\mu(Q_-)=\frac{d_-}{r_-}>2\gamma-2=-\chi(M).
	\end{align*}
	
	Choose a nonzero smooth real-linear bundle homomorphism
	\begin{align*}
		\Phi:Q_-\longrightarrow\Lambda^{0,1}T^*M\otimes Q_+,
	\end{align*}
	and define
	\begin{align*}
		A=
		\begin{pmatrix}
			0&\Phi\\
			0&0
		\end{pmatrix},
		\quad
		D=
		\begin{pmatrix}
			\ddbar_{Q_+}&\Phi\\
			0&\ddbar_{Q_-}
		\end{pmatrix}.
	\end{align*}
	Then $A\neq0$, while
	\begin{align*}
		A(Q_+)=0,
		\quad
		A(E)\subset\Lambda^{0,1}T^*M\otimes Q_+.
	\end{align*}
	Hence $A$ strictly lowers the Harder--Narasimhan filtration. The induced operators on the successive quotients are simply
	\begin{align*}
		D_1=\ddbar_{Q_+},
		\quad
		D_2=\ddbar_{Q_-}.
	\end{align*}
	Since
	\begin{align*}
		\mu(Q_+)>\mu(Q_-)>-\chi(M),
	\end{align*}
	$D$ is surjective by \Cref{cor:strict-HN}.
\end{example}

\section{Applications to automatic transversality}
We now translate the abstract surjectivity results into concrete automatic transversality criteria for pseudoholomorphic curves, including their behavior under pseudoholomorphic immersions.

Let $u:(M,j)\to(X,J)$ be a pseudoholomorphic curve and let
\begin{align*}
	D_u=\ddbar_u+A_u
\end{align*}
be its linearized Cauchy--Riemann operator. Recall from \Cref{sec:2.1} that its anti-complex-linear part satisfies
\begin{align*}
	A_u=\frac14\mathcal N_{J,u},
\end{align*}
where
\begin{align*}
	\mathcal N_{J,u}:=N_J(\,\cdot\,,{\rm d}u(\,\cdot\,)): u^*TX\longrightarrow\Lambda^{0,1}T^*M\otimes u^*TX,
\end{align*}
and $N_J$ denotes the Nijenhuis tensor of $J$. 

If $u$ is immersed, let $\pi^{\mathrm N}:u^*TX\to N_u$ be the quotient map. Taking the anti-complex-linear parts of \eqref{eq:normal} gives
\begin{align*}
	A_u^{\mathrm N}\pi^{\mathrm N}=({\Id}\otimes\pi^{\mathrm N})A_u,
\end{align*}
and hence
\begin{align*}
	A_u^{\mathrm N}=\frac14\mathcal N_{J,u}^{\mathrm N},
\end{align*}
where $\mathcal N_{J,u}^{\mathrm N}$ is characterized by
\begin{align}\label{eq:Nij}
	\mathcal N_{J,u}^{\mathrm N}\pi^{\mathrm N}=({\Id}\otimes\pi^{\mathrm N})\mathcal N_{J,u}.
\end{align}

Thus the preceding surjectivity criteria can be expressed directly in terms of the Nijenhuis tensor along $u$.
\begin{corollary}\label{cor:Jcurve}
	Let $u:(M,j)\to(X,J)$ be a closed pseudoholomorphic curve. Then $u$ is Fredholm regular if any one of the following conditions holds:
	
	\begin{enumerate}[label=\textup{(\roman*)}]
		\item There exists a Hermitian metric $h$ on $u^*TX$ such that
		\begin{align*}
			\int_M\theta_{u^*TX,h,g}\,dV_g+2\pi\chi(M)>\frac1{16}\int_M|\mathcal N_{J,u}|_{g,h}^2\,dV_g.
		\end{align*}
		
		\item There exists a sequence of Hermitian metrics $h_\nu$ on $u^*TX$ such that
		\begin{align*}
			\int_M\theta_{u^*TX,h_\nu,g}\,dV_g
			\longrightarrow
			2\pi\mu_{\min}^{\HN}(u^*TX)
		\end{align*}
		and
		\begin{align*}
			\limsup_{\nu\to\infty}\int_M|\mathcal N_{J,u}|_{g,h_\nu}^2\,dV_g<32\pi\bigl(\mu_{\min}^{\HN}(u^*TX)+\chi(M)\bigr).
		\end{align*}
		
		\item There exists a holomorphic filtration
		\begin{align*}
			0=E_0\subset E_1\subset\cdots\subset E_\ell=u^*TX
		\end{align*}
		such that $\mathcal N_{J,u}(E_i)\subseteq\Lambda^{0,1}T^*M\otimes E_i$ for $1\leq i\leq\ell$. Writing $Q_i:=E_i/E_{i-1}$, assume that
		\begin{align*}
			\mu_{\min}^{\HN}(Q_i)>-\chi(M),\quad 1\leq i\leq\ell,
		\end{align*}
		and that for each $i$, one of the following holds:
		\begin{enumerate}[label=\textup{(\alph*)}]
			\item $\mathcal N_{J,u}(E_i)\subseteq\Lambda^{0,1}T^*M\otimes E_{i-1}$;
			\item $Q_i$ is a holomorphic line bundle.
		\end{enumerate}
	\end{enumerate}
	
	If $u$ is immersed, the same three criteria apply to the normal operator $D_u^{\mathrm N}$ on $N_u$, with $u^*TX$ and $\mathcal N_{J,u}$ replaced by $N_u$ and $\mathcal N_{J,u}^{\mathrm N}$, respectively.
\end{corollary}

\begin{remark}
	For an immersed pseudoholomorphic sphere $u:\mathbb{CP}^1\to(X,J)$, Fredholm regularity is equivalent to the surjectivity of the operator $D_u$. Indeed, Fredholm regularity is equivalent to the surjectivity of the normal operator $D_u^{\mathrm N}$, so it remains to show that the surjectivity of $D_u^{\mathrm N}$ implies that of $D_u$. The operator $D_u$ preserves the holomorphic filtration
	\begin{align*}
		0\subset {\rm d}u(T\mathbb{CP}^1)\subset u^*TX.
	\end{align*}
	Since
	\begin{align*}
		A_u({\rm d}u(T\mathbb{CP}^1))=\frac14\mathcal N_{J,u}({\rm d}u(T\mathbb{CP}^1))=0,
	\end{align*}
	the restriction of $D_u$ to ${\rm d}u(T\mathbb{CP}^1)$ is the complex-linear operator $\ddbar_u$, and via ${\rm d}u$ it is naturally identified with the standard Cauchy--Riemann operator $\ddbar_{T\mathbb{CP}^1}$. As $T\mathbb{CP}^1\cong\mathcal O(2)$, this induced operator is surjective. The operator induced on the quotient $u^*TX/{\rm d}u(T\mathbb{CP}^1)=N_u$ is precisely $D_u^{\mathrm N}$. Hence \Cref{prop:filtered} shows that the surjectivity of $D_u^{\mathrm N}$ implies that of $D_u$.
\end{remark}

\begin{remark}
(i) Assume that $(X,J)$ is K\"ahler. Then $N_J=0$, and hence $A_u=0$. The hypothesis of \Cref{cor:Jcurve} (i) therefore becomes
\begin{align*}
	\int_M\theta_{u^*TX,h,g}\,dV_g>-2\pi\chi(M),
\end{align*}
which agrees with the sufficient condition in \cite[Corollary 1.5]{JiZhu}. Optimizing over $h$ and using \Cref{thm:optimal-curvature} gives the metric-free sufficient condition
\begin{align*}
	\mu_{\min}^{\HN}(u^*TX)>-\chi(M).
\end{align*}
For a semistable pullback bundle, this reduces to
\begin{align*}
	\frac{\deg(u^*TX)}{\rank_{\mathbb C}(u^*TX)}>-\chi(M),
\end{align*}
and for a line bundle it reduces to the classical degree condition.

(ii) Suppose that $\dim_{\mathbb R}X=4$ and $u$ is immersed. Then $N_u$ is a holomorphic line bundle, and the one-step filtration $0\subset N_u$ reduces \Cref{cor:Jcurve} (iii) to
\begin{align*}
\deg N_u>-\chi(M)
\quad\Longrightarrow\quad
D_u^{\mathrm N}\text{ is surjective}.
\end{align*}
This is precisely the closed immersed line-bundle criterion of Hofer--Lizan--Sikorav \cite{HLS}. This rank-one situation has been developed considerably further in the theory of pseudoholomorphic curves. C. Wendl extended automatic transversality in dimension four to punctured curves and to settings that need not be somewhere immersed or injective \cite{WendlAT}. For unbranched covers of closed pseudoholomorphic curves, generic transversality was studied by Gerig--Wendl \cite{GerigWendl}. Wendl later established higher-dimensional transversality and super-rigidity results for multiply covered holomorphic curves \cite{WendlSuperRigidity}.
\end{remark}

We next consider how Fredholm regularity behaves under pseudoholomorphic immersions. Let $f:(X',J')\to(X,J)$ be a pseudoholomorphic immersion, let $u:(M,j)\to(X',J')$ be a closed immersed pseudoholomorphic curve, and set $v:=f\circ u:(M,j)\to(X,J)$. Denote by
\begin{align*}
	N_f:=f^*TX/{\rm d}f(TX')
\end{align*}
the normal bundle of $f$, and write
\begin{align*}
	N_u:=u^*TX'/{\rm d}u(TM),
	\quad
	N_v:=v^*TX/{\rm d}v(TM),
	\quad
	Q:=u^*N_f.
\end{align*}
We denote the corresponding normal operators by $D_u^{\mathrm N}$ and $D_v^{\mathrm N}$. The inclusions
\begin{align*}
	{\rm d}v(TM)\subset{\rm d}f(u^*TX')\subset v^*TX
\end{align*}
determine a natural short exact sequence of complex vector bundles
\begin{align*}
	0\longrightarrow N_u\xrightarrow{\ \iota\ }N_v\xrightarrow{\ q\ }Q\longrightarrow0,
\end{align*}
where $\iota([\cdot])=[{\rm d}f(\cdot)]$ and $q$ is induced by the map $\pi_f:v^*TX\rightarrow u^*N_f$.

For any smooth map $w:M\to X'$, the pseudoholomorphicity of $f$ implies
\begin{align*}
	\ddbar_J(f\circ w)
	=({\Id}\otimes{\rm d}f)\ddbar_{J'}w.
\end{align*}
Linearizing this identity at the pseudoholomorphic map $u$, we obtain
\begin{align}\label{eq:relative-linearization}
	D_v{\rm d}f
	=({\Id}\otimes{\rm d}f)D_u.
\end{align}
Let $\pi_u^{\mathrm N}:u^*TX'\to N_u$ and $\pi_v^{\mathrm N}:v^*TX\to N_v$ be the normal quotient maps. Since $\iota\pi_u^{\mathrm N}=\pi_v^{\mathrm N}{\rm d}f$, \eqref{eq:normal} and \eqref{eq:relative-linearization} imply
\begin{align}\label{eq:relative-normal}
	D_v^{\mathrm N}\iota
	=({\Id}\otimes\iota)D_u^{\mathrm N}.
\end{align}
Taking the complex-linear parts of \eqref{eq:relative-normal} shows that $\iota(N_u)$ is a holomorphic subbundle of $N_v$. Hence $D_v^{\mathrm N}$ preserves the holomorphic filtration
\begin{align*}
	0\subset\iota(N_u)\subset N_v.
\end{align*}
Its restriction to $\iota(N_u)$ is identified via $\iota$ with $D_u^{\mathrm N}$, while the operator induced on the quotient $Q$ is a real-linear Cauchy--Riemann operator
\begin{align*}
	D_Q=\ddbar_Q+A_Q,
\end{align*}
where $\ddbar_Q$ determines the induced holomorphic structure on $Q$. Since $q\pi_v^{\mathrm N}=\pi_f$, the quotient relations for $D_v^{\mathrm N}$ and $D_Q$ imply
\begin{align}\label{eq:DQ}
	D_Q\pi_f
	=({\Id}\otimes q)D_v^{\mathrm N}\pi_v^{\mathrm N}
	=({\Id}\otimes\pi_f)D_v.
\end{align}
Taking the anti-complex-linear parts of \eqref{eq:DQ} and using $A_v=\frac14\mathcal N_{J,v}$, we obtain
\begin{align}\label{eq:relative-A}
	A_Q\pi_f
	=\frac14({\Id}\otimes\pi_f)\mathcal N_{J,v}.
\end{align}

\begin{lemma}\label{lem:relative-normal}
	If $u$ is Fredholm regular in $X'$, then
	\begin{align*}
		v\text{ is Fredholm regular in }X
		\quad\Longleftrightarrow\quad
		D_Q\text{ is surjective}.
	\end{align*}
\end{lemma}

\begin{proof}
	Since $u$ is Fredholm regular in $X'$, the operator $D_u^{\mathrm N}$ is surjective. If $D_Q$ is surjective as well, then \Cref{prop:filtered}, applied to the filtration
	\begin{align*}
		0\subset\iota(N_u)\subset N_v,
	\end{align*}
	implies that $D_v^{\mathrm N}$ is surjective, and hence $v$ is Fredholm regular in $X$. Conversely, if $v$ is Fredholm regular in $X$, then $D_v^{\mathrm N}$ is surjective. Since $D_Qq=({\Id}\otimes q)D_v^{\mathrm N}$, the surjectivity of $D_v^{\mathrm N}$ implies that of $D_Q$.
\end{proof}

\begin{corollary}\label{cor:relative-automatic}
	Let $f:(X',J')\to(X,J)$ be a pseudoholomorphic immersion and let $u:(M,j)\to(X',J')$ be a closed immersed pseudoholomorphic curve which is Fredholm regular in $X'$. Then $v=f\circ u$ is Fredholm regular in $X$ if either of the following conditions holds:
	\begin{enumerate}[label=\textup{(\roman*)}]
		\item $\rank_{\mathbb C}(N_f)=1$ and
		\begin{align*}
			\deg(u^*N_f)>-\chi(M).
		\end{align*}
		\item $N_J\bigl(v^*TX,{\rm d}v(TM)\bigr)\subset{\rm d}f(u^*TX')$ and
		\begin{align*}
			\mu_{\min}^{\HN}(u^*N_f)>-\chi(M).
		\end{align*}
	\end{enumerate}
\end{corollary}

\begin{proof}
	Set $Q:=u^*N_f$. By \Cref{lem:relative-normal}, it suffices to prove that $D_Q$ is surjective. Under (i), $Q$ is a holomorphic line bundle with $\deg Q>-\chi(M)$, so \Cref{thm:HLSline} implies that $D_Q$ is surjective. Under (ii), the Nijenhuis condition implies $({\Id}\otimes\pi_f)\mathcal N_{J,v}=0$. Hence \eqref{eq:relative-A} and the surjectivity of $\pi_f$ imply that $A_Q=0$, and hence $D_Q=\ddbar_Q$. The condition $\mu_{\min}^{\HN}(Q)>-\chi(M)$ then ensures the surjectivity of $D_Q$ by \Cref{cor:strict-HN}. In either case, \Cref{lem:relative-normal} shows that $v$ is Fredholm regular in $X$.
\end{proof}

We conclude this section by explaining how the examples constructed in previous sections arise as normal operators of embedded pseudoholomorphic curves. We first recall the following realization principle. Let $E\to M$ be a complex vector bundle over a closed Riemann surface and let
\begin{align*}
	D=\ddbar_E+A:C^\infty(M,E)\longrightarrow\Omega^{0,1}(E)
\end{align*}
be a real-linear Cauchy--Riemann operator. By \cite[Proposition 3.15]{DoanWalpuski}, there exists a homogeneous almost complex structure $J_D$ on the total space
\begin{align*}
	\operatorname{Tot}(E)
\end{align*}
such that, under the canonical vertical identification, the operator
\begin{align*}
	\mathcal F_D(\xi):=\frac12\bigl({\rm d}\xi+J_D\circ{\rm d}\xi\circ j\bigr)
\end{align*}
satisfies
\begin{align}\label{eq:realization-section}
	\mathcal F_D(\xi)=D\xi
\end{align}
for every $\xi\in C^\infty(M,E)$. Consequently,
\begin{align*}
	D\xi=0
	\quad\Longleftrightarrow\quad
	\xi:(M,j)\longrightarrow(\operatorname{Tot}(E),J_D)\text{ is pseudoholomorphic}.
\end{align*}
In particular, the zero section
\begin{align*}
	u:M&\longrightarrow \operatorname{Tot}(E),\\
	x&\longmapsto0_x,
\end{align*}
is an embedded $J_D$-holomorphic curve. Along $u$, the vertical tangent bundle identifies canonically with the normal bundle,
\begin{align*}
	N_{u}\cong E.
\end{align*}
Under this identification, vertical variations of $u$ are precisely
sections of $E$. For $\xi\in C^\infty(M,E)$, consider the family of
sections
\begin{align*}
	u_t=t\xi.
\end{align*}
Then $u_0=u$ and $\dot u_0=\xi$ as a normal variation. Linearizing
\eqref{eq:realization-section} at $t=0$ gives
\begin{align*}
	D_u^{\mathrm N}\xi=
	\left.\frac{d}{dt}\right|_{t=0}\mathcal F_D(t\xi)=\left.\frac{d}{dt}\right|_{t=0}D(t\xi)=D\xi.
\end{align*}
Hence, under the identification $N_u\cong E$, we have
\begin{align*}
	D_u^{\mathrm N}=D.
\end{align*}

Consequently, each real-linear Cauchy--Riemann operator constructed in \Cref{ex:split-curvature}, \Cref{ex:HN-extension}, \Cref{ex:line-flag} and \Cref{ex:HN-flag} admits a geometric realization as the normal operator of an embedded pseudoholomorphic curve. More precisely, applying the preceding construction to each such operator $D$ produces an embedded pseudoholomorphic zero section $u:M\to X:=\operatorname{Tot}(E)$ satisfying
\begin{align*}
	N_u\cong E,
	\quad
	D_u^{\mathrm N}=D.
\end{align*}

\section{Applications to pseudoholomorphic spheres in \texorpdfstring{$S^6$}{S6}}\label{sec:S6}

This section applies the preceding surjectivity criteria to immersed pseudoholomorphic spheres in $S^6$. We first treat an arbitrary almost complex structure under a transverse automorphism hypothesis and then specialize to the standard octonionic nearly K\"ahler structure, where the relevant holomorphic normal bundle splitting admits a geometric characterization.

Given an almost complex structure $J$ on $S^6$, let
\begin{align*}
	u:\mathbb{CP}^1\longrightarrow (S^6,J)
\end{align*}
be an immersed pseudoholomorphic sphere.
Its normal bundle $N_u$ has complex rank two, and the normal operator has the form
\begin{align*}
	D_u^{\mathrm N}=\ddbar_u^{\mathrm N}+\frac14\mathcal N_{J,u}^{\mathrm N},
\end{align*}
where $\mathcal N_{J,u}^{\mathrm N}$ is characterized by \eqref{eq:Nij}.

\begin{proposition}\label{prop:S6-degree}
	For every immersed pseudoholomorphic sphere $u:\mathbb{CP}^1\to(S^6,J)$,
	\begin{align*}
		\deg N_u=-2,
		\quad
		\mu(N_u)=-1,
	\end{align*}
	and
	\begin{align}\label{eq:S6-index}
		\operatorname{ind}_{\mathbb R}D_u^{\mathrm N}=0.
	\end{align}
\end{proposition}

\begin{proof}
	Since $H^2(S^6;\mathbb Z)=0$, one has $c_1(TS^6,J)=0$. The immersion gives an exact sequence of complex vector bundles
	\begin{align*}
		0\longrightarrow T\mathbb{CP}^1\xrightarrow{\,{\rm d}u\,}u^*TS^6\longrightarrow N_u\longrightarrow0.
	\end{align*}
	Hence
	\begin{align*}
		c_1(N_u)=u^*c_1(TS^6,J)-c_1(T\mathbb{CP}^1)=-c_1(T\mathbb{CP}^1),
	\end{align*}
	which implies
	\begin{align*}
		\deg N_u=-2\ \ \text{and}\ \ \mu(N_u)=-1.
	\end{align*}
	The Riemann--Roch formula for a real-linear Cauchy--Riemann operator on the rank-two bundle $N_u$ then yields
	\begin{equation*}
		\operatorname{ind}_{\mathbb R}D_u^{\mathrm N}
		=2\deg N_u+2\chi(\mathbb{CP}^1)
		=-4+4=0.
		\tag*{\qedhere}
	\end{equation*}
\end{proof}
	
The Birkhoff--Grothendieck theorem \cite{Grothendieck} gives
\begin{align*}
	N_u\cong\mathcal O(a)\oplus\mathcal O(b),
	\quad
	a\geq b,
\end{align*}
where $a+b=-2$ by \Cref{prop:S6-degree}.

\begin{corollary}[=\Cref{cor:S6-symmetry'}]\label{cor:S6-symmetry}
	Let $J$ be an almost complex structure on $S^6$, and let $u:\mathbb{CP}^1\to(S^6,J)$ be an immersed pseudoholomorphic sphere satisfying
	\begin{align}\label{eq:S6-splitting1}
		N_u\cong\mathcal O(-1)\oplus\mathcal O(-1).
	\end{align}
	Suppose that there exists $z_0\in\mathbb{CP}^1$ such that
	\begin{align}\label{eq:S6-orbit}
		T_{u(z_0)}\bigl(\operatorname{Aut}(S^6,J)\cdot u(z_0)\bigr)\not\subseteq {\rm d}u_{z_0}(T_{z_0}\mathbb{CP}^1).
	\end{align}
	Let $h$ be a Hermitian--Einstein metric on $N_u$. Then
	\begin{enumerate}[label=\textup{(\roman*)}]
		\item The induced normal Nijenhuis term satisfies
		\begin{align}\label{eq:S6-energy}
			\int_{\mathbb{CP}^1}
			|\mathcal N_{J,u}^{\mathrm N}|_{g,h}^2\,dV_g
			\geq32\pi.
		\end{align}
		In particular, $J$ cannot be integrable.
		\item There is no holomorphic line subbundle $L\cong\mathcal O(-1)\subset N_u$ that is preserved by $\mathcal N_{J,u}^{\mathrm N}$.
	\end{enumerate}
\end{corollary}

\begin{proof}
	Hypothesis \eqref{eq:S6-orbit} yields an element
	$Z\in\mathfrak{aut}(S^6,J)$ whose fundamental vector field $Z^\sharp$ satisfies
	\begin{align}\label{eq:S6-transverse}
		Z^\sharp(u(z_0))
		\notin {\rm d}u_{z_0}(T_{z_0}\mathbb{CP}^1),
	\end{align}
	where $\mathfrak{aut}(S^6,J)$ denotes the Lie algebra of the pseudoholomorphic automorphism group $\operatorname{Aut}(S^6,J)$. Hence the family
	\begin{align*}
		u_t=\exp(tZ)\circ u
	\end{align*}
	consists of pseudoholomorphic maps. Differentiating the equation $\ddbar_Ju_t=0$ at $t=0$ gives
	\begin{align*}
		Z^\sharp\circ u\in\ker D_u.
	\end{align*}
	Passing to the normal quotient therefore yields
	\begin{align*}
		\sigma_{Z}:=\pi^{\mathrm N}(Z^\sharp\circ u)\in\ker D_u^{\mathrm N},
	\end{align*}
	where $\pi^{\mathrm N}:u^*TS^6\rightarrow N_u$ is the quotient map. Moreover, \eqref{eq:S6-transverse} implies $\sigma_Z(z_0)\neq0$, so $\sigma_Z$ is a nonzero element of $\ker D_u^{\mathrm N}$. In view of \eqref{eq:S6-index},
	\begin{align*}
		\dim_{\mathbb R}{\coker}D_u^{\mathrm N}=\dim_{\mathbb R}\ker D_u^{\mathrm N}\geq1.
	\end{align*}
	Thus $D_u^{\mathrm N}$ is not surjective, i.e., $u$ is not Fredholm regular.
	
	We first prove (i). Since $h$ is Hermitian--Einstein, the contracted Chern curvature is scalar, and the Chern--Weil formula gives
	\begin{align*}
		\int_{\mathbb{CP}^1}\theta_{N_u,h,g}\,dV_g=-2\pi.
	\end{align*}
	If the integral in \eqref{eq:S6-energy} were strictly smaller than $32\pi$, then, since $\chi(\mathbb{CP}^1)=2$,
	\begin{align*}
		\int_{\mathbb{CP}^1}\theta_{N_u,h,g}\,dV_g+2\pi\chi(\mathbb{CP}^1)=2\pi>\frac1{16}
		\int_{\mathbb{CP}^1}
		|\mathcal N_{J,u}^{\mathrm N}|_{g,h}^2\,dV_g.
	\end{align*}
	By \Cref{cor:Jcurve}, $u$ is Fredholm regular, contradicting the preceding paragraph. This proves \eqref{eq:S6-energy}. Since integrability of $J$ would imply $N_J\equiv0$ and hence $\mathcal N_{J,u}^{\mathrm N}\equiv0$, it also follows that $J$ is not integrable.
	
	For (ii), suppose that $\mathcal N_{J,u}^{\mathrm N}$ preserves a holomorphic line subbundle $L\cong\mathcal O(-1)\subset N_u$. Then $N_u/L$ is a holomorphic line bundle and
	\begin{align*}
		\deg(N_u/L)=\deg N_u-\deg L=-1.
	\end{align*}
	Thus the holomorphic filtration
	\begin{align}\label{eq:fil}
		0\subset L\subset N_u
	\end{align}
	has successive quotients of degree $-1$. In particular,
	\begin{align*}
		\mu_{\min}^{\HN}(L)=\mu_{\min}^{\HN}(N_u/L)=-1>-2=-\chi(\mathbb{CP}^1).
	\end{align*}
	Since $\mathcal N_{J,u}^{\mathrm N}$ preserves $L$, the operator $D_u^{\mathrm N}$ preserves the filtration \eqref{eq:fil}. Therefore \Cref{cor:Jcurve} implies that $u$ is Fredholm regular, again a contradiction. Hence no such $\mathcal N_{J,u}^{\mathrm N}$-invariant holomorphic line subbundle exists.
\end{proof}

\begin{remark}
	The argument is not specific to $S^6$. More generally, let $(X,J)$ be a compact almost complex manifold of real dimension $6$, and let $u:\mathbb{CP}^1\to(X,J)$ be an immersed pseudoholomorphic sphere satisfying
	\begin{align*}
		\left\langle c_1(TX,J),u_*[\mathbb{CP}^1]\right\rangle=0.
	\end{align*}
	Then $\deg N_u=-2$ and $\operatorname{ind}_{\mathbb R}D_u^{\mathrm N}=0$. Under the same splitting and transverse-orbit hypotheses as in \Cref{cor:S6-symmetry}, the argument applies verbatim, with $\operatorname{Aut}(S^6,J)$ replaced by $\operatorname{Aut}(X,J)$.
\end{remark}

We now specialize to the standard almost complex structure $J_0$ on $S^6\subset\operatorname{Im}\mathbb O$, defined by
\begin{align*}
	(J_0)_x(v)=x\times v,\quad x\in S^6,\quad v\in T_xS^6=x^\perp,
\end{align*}
where $\times$ is the octonionic cross product. Together with the round metric $g_0$, this is the \textit{standard nearly K\"ahler structure} on $S^6$ (see \cite{Gray1966,FernandezS6}). The standard nearly K\"ahler structure $(g_0,J_0)$ admits a unique
metric connection $\nabla^c$ preserving $J_0$ and having totally
skew-symmetric torsion $T^c$. Thus
\begin{align*}
	\nabla^c g_0=0,
	\quad
	\nabla^c J_0=0.
\end{align*}

Let $u:\mathbb{CP}^1\to(S^6,J_0)$ be an immersed pseudoholomorphic sphere. Since $\nabla^cJ_0=0$, the complexification of $\nabla^c$ preserves $T^{1,0}S^6$, and its pullback along $u$ defines a connection, again denoted by $\nabla^c$, on $u^*T^{1,0}S^6$. We denote by $(\nabla^c)^{0,1}$ its $(0,1)$-part with respect to the complex structure on $\mathbb{CP}^1$. We shall identify the holomorphic structure defined by the complex-linear part of the normal operator with the quotient holomorphic structure appearing in Bryant's holomorphic moving-frame construction \cite[Section 4]{Bryant}.

\begin{lemma}\label{lem:S6-normal}
	Let $u:\mathbb{CP}^1\to (S^6,J_0)$ be an immersed pseudoholomorphic sphere. Then the natural complex-bundle isomorphism
	\begin{align}\label{eq:iso}
		(N_u,J_0)\cong u^*T^{1,0}S^6/{\rm d}u(T^{1,0}\mathbb{CP}^1)
	\end{align}
	identifies the holomorphic structure defined by $\ddbar_u^{\mathrm N}$ with the quotient holomorphic structure induced by $(\nabla^c)^{0,1}$.
\end{lemma}

\begin{proof}
	We first identify the complex-linear part of the linearized Cauchy--Riemann operator. Since $\nabla^cJ_0=0$, we may apply \eqref{eq:linearization}. Its torsion $T^c$ is a nonzero constant multiple of the Nijenhuis tensor $N_{J_0}$ \cite[Corollary 10.3]{FriedrichIvanov}. Since
	\begin{align*}
		N_{J_0}(J_0V,W)=N_{J_0}(V,J_0W)=-J_0N_{J_0}(V,W),
	\end{align*}
	the torsion term $B_u$ in \eqref{eq:linearization} is anti-complex-linear in $\xi$. Hence
	\begin{align}\label{eq:S6-dbar}
		\ddbar_u=(\nabla^c)^{0,1}.
	\end{align}
	Combining \eqref{eq:S6-dbar} with \eqref{eq:normal-dbar} gives
	\begin{align}\label{eq:equ}
		\ddbar_u^{\mathrm N}\pi^{\mathrm N}=({\Id}\otimes\pi^{\mathrm N})(\nabla^c)^{0,1},
	\end{align}
	where $\pi^{\mathrm N}:u^*TS^6\to N_u$ is the quotient map.
	
	It remains to identify the quotient holomorphic structure. The complex-linear map
	\begin{align*}
		P^{1,0}:(u^*TS^6,J_0)&\longrightarrow u^*T^{1,0}S^6,\\
		v&\longmapsto \frac12\bigl(v-{\ii}J_0v\bigr)
	\end{align*}
	identifies $(u^*TS^6,J_0)$ with $u^*T^{1,0}S^6$ and, since $u$ is pseudoholomorphic, maps ${\rm d}u(T\mathbb{CP}^1)$ onto ${\rm d}u(T^{1,0}\mathbb{CP}^1)$. It therefore descends to the natural complex-bundle isomorphism \eqref{eq:iso}. Moreover, the map $P^{1,0}$ is parallel with respect to $\nabla^c$ and hence intertwines the induced $(0,1)$-operators. In view of \eqref{eq:S6-dbar} and \eqref{eq:equ}, the descended map therefore intertwines $\ddbar_u^{\mathrm N}$ with the quotient Cauchy--Riemann operator induced by $(\nabla^c)^{0,1}$. Hence \eqref{eq:iso} is an isomorphism of holomorphic vector bundles.
\end{proof}

The condition \eqref{eq:S6-splitting1} has a concrete geometric meaning for the standard nearly K\"ahler structure.

\begin{proposition}\label{prop:S6-TG}
	Let $u:\mathbb{CP}^1\to(S^6,J_0)$ be an immersed pseudoholomorphic sphere. Then the following are equivalent:
	\begin{enumerate}[label=\textup{(\roman*)}]
		\item $u$ is totally geodesic with respect to the round metric;
		\item the holomorphic normal bundle $N_u$ satisfies
		\begin{align*}
			N_u\cong\mathcal O(-1)\oplus\mathcal O(-1).
		\end{align*}
	\end{enumerate}
\end{proposition}

\begin{proof}
	$(i)\Rightarrow(ii)$. Assume that $u$ is totally geodesic. Then its image is a great $2$-sphere $P\cap S^6$ for some three-plane $P\subset\operatorname{Im}\mathbb O$. Since $u$ is pseudoholomorphic, $P$ is associative. By the transitivity of the $G_2$-action on associative three-planes \cite[\S IV.1.A, Theorem 1.8]{HarveyLawson}, after applying an element of $G_2$ and a biholomorphic reparametrization of $\mathbb{CP}^1$, we may assume that
	\begin{align*}
		u:S^2=S^6\cap\operatorname{Im}\mathbb H\hookrightarrow S^6
	\end{align*}
	is the standard associative equator.
	
	Write
	\begin{align*}
		\mathbb O=\mathbb H\oplus\mathbb H\varepsilon,
		\quad
		(a+b\varepsilon)(c+d\varepsilon)
		=
		(ac-\overline d\,b)+(da+b\overline c)\varepsilon.
	\end{align*}
	Using the round metric, we identify $N_u$ with the orthogonal normal bundle, which along $S^2$ is the fixed real summand $\mathbb H\varepsilon$. Thus a normal section can be written as $\eta=a\varepsilon$, where $a:S^2\to\mathbb H$. Since
	\begin{align*}
		J_{0,x}(a\varepsilon)
		=
		x(a\varepsilon)
		=
		(ax)\varepsilon,
	\end{align*}
	the complex structure on the coefficient space $\mathbb H$ is
	\begin{align*}
		I_xa=ax.
	\end{align*}
	
	Let $\nabla^{g_0}$ denote the Levi--Civita connection of the round metric. Then
	\begin{align*}
		\nabla_V^cW=\nabla_V^{g_0}W-\frac12J_0(\nabla_V^{g_0}J_0)W.
	\end{align*}
	For $Y\in T_xS^2$ and $\eta=a\varepsilon$, one has
	\begin{align*}
		\nabla_Y^{g_0}(a\varepsilon)=Y(a)\varepsilon,
		\quad
		(\nabla_Y^{g_0}J_0)(a\varepsilon)=Y\times(a\varepsilon)=(aY)\varepsilon.
	\end{align*}
	Consequently,
	\begin{align}\label{eq:characteristic-explicit}
		\nabla_Y^c(a\varepsilon)=\left(Y(a)-\frac12aYx\right)\varepsilon.
	\end{align}
	In particular, $\nabla^c$ preserves the normal summand $\mathbb H\varepsilon$ along the associative equator.
	
	Choose a local conformal coordinate $z=s+{\ii}t$ with $j\partial_s=\partial_t$, and write $x=u(s,t)$. Pseudoholomorphicity gives
	\begin{align*}
		x_t=xx_s,
	\end{align*}
	and hence, since $x^2=-1$,
	\begin{align*}
		x_sx=-x_t,
		\quad
		x_tx=x_s.
	\end{align*}
	Applying \eqref{eq:characteristic-explicit} with $Y=x_s$ and $Y=x_t$ gives
	\begin{align*}
		\nabla_{\partial_s}^c(a\varepsilon)=\left(a_s+\frac12ax_t\right)\varepsilon,\quad
		\nabla_{\partial_t}^c(a\varepsilon)=\left(a_t-\frac12ax_s\right)\varepsilon.
	\end{align*}
	By \Cref{lem:S6-normal}, $\ddbar_u^{\mathrm N}$ is the quotient Cauchy--Riemann operator induced by $(\nabla^c)^{0,1}$. Since the expressions above are already normal, we obtain
	\begin{align}\label{eq:normal-explicit}
		\ddbar_u^{\mathrm N}(a\varepsilon)(\partial_s)=\frac12\left(a_s+\frac12ax_t+\left(a_t-\frac12ax_s\right)x\right)\varepsilon=\frac12\left(a_s+a_tx+ax_t\right)\varepsilon.
	\end{align}
	
	We now identify the resulting holomorphic bundle explicitly. Using the standard matrix representation of $\mathbb H$, its complexification gives an identification
	\begin{align*}
		\mathbb H\otimes_{\mathbb R}\mathbb C\cong M_2(\mathbb C),
	\end{align*}
	under which imaginary quaternions are represented by traceless skew-Hermitian matrices. We use the same symbol $x$ for the matrix corresponding to $x\in S^2$. The ${\ii}$-eigenspace of the complexification of $I_x$ is
	\begin{align*}
		N_{u,x}^{1,0}=\left\{B\in M_2(\mathbb C)\ |\ Bx={\ii}B\right\},
	\end{align*}
	where we use the natural identification $N_u^{\mathbb C}\cong S^2\times M_2(\mathbb C)\varepsilon\cong S^2\times M_2(\mathbb C)$. For a section $\sigma$ of the bundle $N_u^{1,0}$, differentiating $\sigma x={\ii}\sigma$ with respect to $t$ gives 
	\begin{align*} 
		\sigma_tx+\sigma x_t={\ii}\sigma_t. 
	\end{align*} 
	Hence \eqref{eq:normal-explicit} reduces to $\ddbar_u^{\mathrm N}\sigma(\partial_s)=\frac12(\sigma_s+{\ii}\sigma_t) =\partial_{\overline z}\sigma$, i.e.,
	\begin{align}\label{eq:normal-standard} 
		\ddbar_u^{\mathrm N}\sigma= \overline{\partial}\sigma. 
	\end{align}
	
	Under the above matrix realization, the stereographic coordinate $z$ on $U_0=\mathbb{CP}^1\setminus\{\infty\}$ gives
	\begin{align*} 
		x(z)= \frac{{\ii}}{1+|z|^2} 
		\begin{pmatrix} 
			1-|z|^2 & 2z\\ 
			2\overline z & |z|^2-1 
		\end{pmatrix}. 
	\end{align*} 
	The row vector
	\begin{align*}
		r_0(z)=(1,z)
	\end{align*}
	satisfies $r_0(z)x(z)={\ii}\,r_0(z)$ and spans the ${\ii}$ left eigenspace of $x(z)$. Thus, if $e_1,e_2$ denote the standard column basis of $\mathbb C^2$, the sections
	\begin{align*}
		\sigma_0^{(j)}=e_jr_0,
		\quad j=1,2,
	\end{align*}
	form a local $\ddbar_u^{\mathrm N}$-holomorphic frame by \eqref{eq:normal-standard}, since $r_0$ depends holomorphically on $z$. On $U_1=\mathbb{CP}^1\setminus\{0\}$, with $w=z^{-1}$, the same argument applied to 
	\begin{align*} 
		r_1(w)=(w,1), 
		\quad 
		\sigma_1^{(j)}=e_jr_1,\quad j=1,2,
	\end{align*}
	gives a $\ddbar_u^{\mathrm N}$-holomorphic frame. On $U_0\cap U_1$, 
	\begin{align*} 
		r_0(z)=z\,r_1(w), 
	\end{align*} 
	and therefore 
	\begin{align*} 
		\sigma_0^{(j)}=z\sigma_1^{(j)}, 
		\quad j=1,2. 
	\end{align*} 
	This is the standard frame transition of the tautological line bundle $\mathcal O(-1)$ in each factor. Hence 
	\begin{align*} 
		(N_u,\ddbar_u^{\mathrm N}) \cong\mathcal O(-1)\oplus\mathcal O(-1) 
	\end{align*} 
	as holomorphic vector bundles.
	
	(ii)$\Rightarrow$(i). Suppose that
	\begin{align*}
		N_u\cong\mathcal O(-1)\oplus\mathcal O(-1).
	\end{align*}
	By Bryant \cite[Lemma 4.3]{Bryant}, the second fundamental form determines a holomorphic section
	\begin{align*}
		\Phi_{\mathrm{II}}
		\in
		H^0\!\left(\mathbb{CP}^1,\operatorname{Hom}(T^{1,0}\mathbb{CP}^1\otimes T^{1,0}\mathbb{CP}^1,u^*T^{1,0}S^6/{\rm d}u(T^{1,0}\mathbb{CP}^1))\right),
	\end{align*}
	and $\Phi_{\mathrm{II}}\equiv0$ if and only if $u$ is totally geodesic \cite[Lemma 4.4]{Bryant}. It follows from
	$T^{1,0}\mathbb{CP}^1\cong\mathcal O(2)$ and \Cref{lem:S6-normal} that
	\begin{align*}
		H^0\!\Big(\mathbb{CP}^1,\operatorname{Hom}(T^{1,0}\mathbb{CP}^1\otimes T^{1,0}\mathbb{CP}^1,u^*T^{1,0}S^6/{\rm d}u(T^{1,0}\mathbb{CP}^1))\Big)\cong H^0\!\left(\mathbb{CP}^1,\mathcal O(-5)\oplus\mathcal O(-5)\right)=0.
	\end{align*}
	Hence $\Phi_{\mathrm{II}}=0$, and therefore $u$ is totally geodesic. 
\end{proof}

\Cref{prop:S6-TG} shows that the hypothesis \eqref{eq:S6-splitting1} holds precisely for totally geodesic pseudoholomorphic spheres.
\begin{corollary}\label{cor:S6-TG}
	Let $u:\mathbb{CP}^1\to(S^6,J_0)$ be a totally geodesic
	pseudoholomorphic sphere. Let $h$ be a Hermitian--Einstein metric on $N_u$. Then $\int_{\mathbb{CP}^1}|\mathcal N_{J_0,u}^{\mathrm N}|_{g,h}^2\,dV_g\geq32\pi$, and no holomorphic line subbundle $L\cong\mathcal O(-1)\subset N_u$ is invariant under $\mathcal N_{J_0,u}^{\mathrm N}$.
\end{corollary}

\begin{proof}
	The group $G_2$ acts transitively on $(S^6,J_0)$, so the transverse-orbit hypothesis \eqref{eq:S6-orbit} is satisfied. The conclusions then follow from \Cref{cor:S6-symmetry}.
\end{proof}

\subsection*{Acknowledgments}
The authors would like to thank Prof. Ke Zhu for helpful discussions on the automatic transversality of pseudoholomorphic curves.

\section*{Statements and Declarations}

\noindent\textbf{Funding.} This work was partially supported by the National Natural Science Foundation of China, Nos. 12431005, 12571086, and 12601141.

\medskip
\noindent\textbf{Competing Interests.} On behalf of all authors, the corresponding author states that there is no conflict of interest.

\medskip
\noindent\textbf{Data Availability.} Data sharing is not applicable to this article as no datasets were generated or analyzed during the current study.

\end{document}